\documentclass[10.5pt,a4paper]{article}

\usepackage{graphicx,latexsym,euscript,makeidx,color,bm}
\usepackage{amsmath,amsfonts,amssymb,amsthm,thmtools,mathrsfs,enumerate}
\usepackage[shortlabels]{enumitem}
\usepackage[colorlinks,linkcolor=blue,anchorcolor=green,citecolor=red]{hyperref}

\usepackage{geometry}
    \def\BA{{\bm A}} 
    \def\BB{{\bm B}} 
     
   \def\cD{{\cal D}}  
\def\dbE{\mathbb{E}}     
\def\dbF{\mathbb{F}}   \def\cF{{\cal F}}  
     
\def\dbH{\mathbb{H}}

\def\dbL{\mathbb{L}}   \def\cL{{\cal L}}  
    \def\BM{{\bm M}} 
   \def\cN{{\cal N}} \def\BN{{\bm N}} 
     
\def\dbP{\mathbb{P}}     
     
\def\dbR{\mathbb{R}}     
\def\dbS{\mathbb{S}}   \def\cS{{\cal S}}  
     
   \def\cU{{\cal U}} \def\BU{{\bm U}}

\def\dbX{\mathbb{X}}   \def\cX{{\cal X}} \def\BX{{\bm X}} 
\def\dbY{\mathbb{Y}}    \def\BY{{\bm Y}} 
\def\dbZ{\mathbb{Z}}    \def\BZ{{\bm Z}} 

\def\ba{\begin{array}}   \def\ea{\end{array}} \def\bel{\begin{equation}\label}  \def\ee{\end{equation}}

\def\ss{\smallskip}      \def\lt{\left}       \def\hb{\hbox}
\def\ms{\medskip}        \def\rt{\right}      \def\ae{\hbox{\rm a.e.}}
\def\bs{\bigskip}        \def\lan{\langle}    \def\as{\hbox{\rm a.s.}}
\def\ds{\displaystyle}   \def\ran{\rangle}    \def\tr{\hbox{\rm tr$\,$}}
\def\ts{\textstyle}      \def\llan{\lt\lan}   
\def\no{\noindent}       \def\rran{\rt\ran}   
\def\ns{\noalign{\ss}}   \def\blan{\big\lan}  \def\esssup{\mathop{\rm ess\,sup}}
   \def\bran{\big\ran}  \def\essinf{\mathop{\rm ess\,inf}}
\def\rf{\eqref}            
         \def\hp{\hphantom}
\def\deq{\triangleq}     \def\({\Big (}       \def\nn{\nonumber}
\def\les{\leqslant}      \def\){\Big )}       
\def\ges{\geqslant}      \def\[{\Big[}        \def\cl{\overline}
          \def\]{\Big]}        
\def\wt{\widetilde}      \def\q{\quad}        
\def\h{\widehat}         \def\qq{\qquad}      \def\1n{\negthinspace}
\def\cd{\cdot}           \def\2n{\1n\1n}      \def\det{\hbox{\rm det\,}}
\def\cds{\cdots}         \def\3n{\1n\2n}

\def\a{\alpha}        \def\G{\Gamma}      \def\Om{\Omega}  \def\om{\omega}
\def\b{\beta}         \def\D{\Delta}   \def\d{\delta}   \def\F{\Phi}     
         \def\Th{\Theta}  \def\th{\theta}  \def\Si{\Sigma}  \def\si{\sigma}
\def\e{\varepsilon}   \def\L{\Lambda}  \def\l{\lambda}  \def\m{\mu}      
    \def\t{\tau}     \def\f{\varphi}  \def\i{\infty}   

\newtheoremstyle{indented}{}{}{\it}{\parindent}{\bfseries}{.}{.5em}{}
\theoremstyle{indented}

\newtheorem{theorem}{Theorem}[section]
\newtheorem{definition}[theorem]{Definition}
\newtheorem{proposition}[theorem]{Proposition}
\newtheorem{corollary}[theorem]{Corollary}
\newtheorem{lemma}[theorem]{Lemma}
\newtheorem{remark}[theorem]{Remark}
\newtheorem{example}[theorem]{Example}

\makeatletter
   
   \@addtoreset{equation}{section}
   \newcommand{\setword}[2]{%
   \phantomsection
   #1\def\@currentlabel{\unexpanded{#1}}\label{#2}%
   }
\makeatother
  \allowdisplaybreaks[4]

\begin{document}

\title{\bf Indefinite Stochastic Linear-Quadratic Optimal Control Problems with Random Coefficients: \\
           Closed-Loop Representation of Open-Loop\\ Optimal Controls}

\author{Jingrui Sun\thanks{Department of Mathematics, Southern University of Science and Technology,
                           Shenzhen, Guangdong, 518055, China (sunjr@sustech.edu.cn).} \and
          Jie Xiong\thanks{Department of Mathematics, Southern University of Science and Technology,
                           Shenzhen, Guangdong, 518055, China (xiongj@sustc.edu.cn).
                           This author is supported by NSFC Grants 61873325 and 11831010, and SUST
                           start-up funds Y01286120 and Y01286220.} \and
      Jiongmin Yong\thanks{Department of Mathematics, University of Central Florida,
                           Orlando, FL 32816, USA. (jiongmin.yong@ucf.edu).
                           This author is supported in part by NSF Grant DMS-1812921.}}

\maketitle

\no{\bf Abstract.}
This paper is concerned with a stochastic linear-quadratic optimal control problem in a finite time horizon, where the coefficients of the control system are allowed to be random, and the weighting matrices in the cost functional are allowed to be random and indefinite. It is shown, with a Hilbert space approach, that for the existence of an open-loop optimal control, the convexity of the cost functional (with respect to the control) is necessary; and the uniform convexity, which is slightly stronger, turns out to be sufficient, which also leads to the unique solvability of the associated stochastic Riccati equation. Further, it is shown that the open-loop optimal control admits a closed-loop representation. In addition, some sufficient conditions are obtained for the uniform convexity of the cost functional, which are strictly general than the classical conditions that the weighting matrix-valued processes are positive (semi-)definite.

\ms

\no{\bf Keywords.} stochastic linear-quadratic optimal control problem, random coefficient, stochastic Riccati equation, value flow, open-loop optimal control, closed-loop representation.

\ms

\no{\bf AMS subject classifications.} 49N10, 49N35, 93E20, 49K45.

\section{Introduction}\label{Sec:Introduction}

Throughout this paper, we let $(\Om,\cF,\dbF,\dbP)$ be a complete filtered probability space on which a standard one-dimensional Brownian motion $W=\{W(t); 0\les t<\i\}$ is defined. We assume that $\dbF=\{\cF_t\}_{t\ges0}$ is the natural filtration of $W$ augmented by all the $\dbP$-null sets in $\cF$. Hence, $\dbF$ automatically satisfies the {\it usual conditions}.

\ms

Consider the following controlled linear stochastic differential equation (SDE, for short) on a finite time horizon:
\bel{state}\left\{\begin{aligned}
   dX(s) &= [A(s)X(s)+B(s)u(s)]ds + [C(s)X(s)+D(s)u(s)]dW(s), \q s\in[t,T],\\
    X(t) &= \xi,
\end{aligned}\right.\ee
where $A,C:[0,T]\times\Om\to\dbR^{n\times n}$ and $B,D:[0,T]\times\Om\to\dbR^{n\times m}$, called the
{\it coefficients} of the {\it state equation} \rf{state}, are given matrix-valued $\dbF$-progressively
measurable processes; and $(t,\xi)$, called an {\it initial pair} (of an initial time and an initial state),
belongs to the following set:
$$ \cD = \big\{(t,\xi) ~|~ t\in\cS[0,T],~\xi\in L^2_{\cF_t}(\Om;\dbR^n)\big\},$$
where $\cS[0,T]$ is the set of all $\dbF$-stopping times valued in $[0,T]$ and
$$ L^2_{\cF_t}(\Om;\dbR^n) = \big\{\xi:\Om\to\dbR^n ~|~ \xi~\hb{is $\cF_t$-measurable with}~\dbE|\xi|^2<\i\big\},\qq t\in[0,T].$$
In the above, the solution $X=\{X(s);t\les s\les T\}$ of \rf{state} is called a {\it state process}; the process $u=\{u(s);t\les s\les T\}$ is called a {\it control} which influences the state $X$, and is taken from the space
\begin{align*}
\cU[t,T] =L^2_\dbF(t,T;\dbR^m)
  &\ts=\Big\{u:[t,T]\times\Om\to\dbR^m\bigm| u~\hb{is $\dbF$-progressively measurable with} \\
  &\ts\hp{=\Big\{u:[t,T]\times\Om\to\dbR^m\bigm|~} \dbE\int_t^T|u(s)|^2ds<\i\Big\}.
\end{align*}
The pair $(X,u)=\{(X(s),u(s));t\les s\les T\}$ is called a {\it state-control pair}
corresponding to the initial pair $(t,\xi)$.
For our state equation \rf{state}, we introduce the following assumption:
\begin{itemize}[leftmargin=0.98cm]
\item[{\bf\setword{(A1)}{(A1)}}] The processes $A,C:[0,T]\times\Om\to\dbR^{n\times n}$ and
$B,D:[0,T]\times\Om\to\dbR^{n\times m}$ are all bounded and $\dbF$-progressively measurable.
\end{itemize}

According to the standard result for SDEs (see \autoref{lmm:sol-F&B-SDE} (i)), under the assumption \ref{(A1)},
for any initial pair $(t,\xi)\in\cD$ and any control $u\in\cU[t,T]$, equation \rf{state} admits a unique solution
$X(\cd)\equiv X(\cd\,;t,\xi,u)$ which has continuous path and is square-integrable.

\medskip

Next we introduce the following random variable associated with the state equation \rf{state}:
\begin{equation}\label{rv-L}
L(t,\xi;u) \deq \lan GX(T),X(T)\ran
                +\int_t^T\llan\1n\begin{pmatrix}Q(s)&\1nS(s)^\top \\ S(s)&\1nR(s)\end{pmatrix}\1n
                                 \begin{pmatrix}X(s) \\ u(s)\end{pmatrix}\1n,
                                 \begin{pmatrix}X(s) \\ u(s)\end{pmatrix}\1n\rran ds,
\end{equation}
where with $\dbS^n$ denoting the set of all symmetric $(n\times n)$ real matrices,
the weighting matrices $G$, $Q$, $S$, and $R$ satisfy the following assumption:
\begin{itemize}[leftmargin=0.98cm]
\item[{\bf\setword{(A2)}{(A2)}}]
The processes $Q:[0,T]\times\Om\to\dbS^n$, $R:[0,T]\times\Om\to\dbS^m$, and $S:[0,T]\times\Om\to\dbR^{m\times n}$
are all bounded and $\dbF$-progressively measurable;
the random variable $G:\Om\to\dbS^n$ is bounded and $\cF_T$-measurable.
\end{itemize}

Under \ref{(A1)}--\ref{(A2)}, the random variable defined by \rf{rv-L} is  integrable,
so the following two functionals are well-defined:
\begin{alignat*}{2}
J(t,\xi;u)    &=\dbE[L(t,\xi;u)];       \q~& (t,\xi)\in\cD,~u\in\cU[t,T], \\
\h J(t,\xi;u) &=\dbE[L(t,\xi;u)|\cF_t]; \q~& (t,\xi)\in\cD,~u\in\cU[t,T].
\end{alignat*}
These two functionals are called the {\it cost functionals} associated with the state equation \rf{state}, which will be used to measure the performance of the control $u\in\cU[t,T]$. Now, the following two problems, called {\it stochastic linear-quadratic optimal control problems} (SLQ problems, for short), can be formulated:

\ms

{\bf Problem (SLQ).} For any given initial pair $(t,\xi)\in\cD$, find a control $u^*\in\cU[t,T]$ such that
\bel{minJ} J(t,\xi;u^*) =\inf_{u\in\cU[t,T]}J(t,\xi;u) \equiv V(t,\xi). \ee

{\bf Problem $\h{\bf(SLQ)}$.} For any given initial pair $(t,\xi)\in\cD$, find a control $u^*\in\cU[t,T]$ such that
\bel{minhJ} \h J(t,\xi;u^*) =\essinf_{u\in\cU[t,T]}\h J(t,\xi;u) \equiv\h V(t,\xi). \ee

\ss

In \rf{minhJ}, $\essinf$ stands for the {\it essential infimum} of a real-valued random variable family. Any element $u^*\in\cU[t,T]$ satisfying \rf{minJ} (respectively, \rf{minhJ}) is called an {\it open-loop optimal control} of Problem (SLQ) (respectively, Problem $\h{\rm(SLQ)}$) corresponding to the initial pair $(t,\xi)\in\cD$;
the corresponding state process $X^*(\cd)\equiv X(\cd\,;t,\xi,u^*)$ is called an {\it open-loop optimal state process};
and the state-control pair $(X^*,u^*)$ is called an {\it open-loop optimal pair} corresponding to $(t,\xi)$.
Since the space $L^2_{\cF_t}(\Om;\dbR^n)$ of initial states becomes larger as the initial time $t$ increases, it is proper to call $(t,\xi)\mapsto V(t,\xi)$ the {\it value flow} of Problem (SLQ) and $(t,\xi)\mapsto\h V(t,\xi)$ the
({\it stochastic}) {\it value flow} of Problem $\h{\rm(SLQ)}$.

\ms

We now introduce the following definition.

\begin{definition}\label{def:ol-opt-cntrl}\rm
Problem (SLQ) (respectively, Problem $\h{\rm(SLQ)}$) is said to be
\begin{enumerate}[\q\,\rm(i)]
\item ({\it uniquely}) {\it open-loop solvable at $(t,\xi)\in\cD$} if there exists a (unique)
      $u^*\in\cU[t,T]$ such that for any $u\in\cU[t,T]$,
      $$J(t,\xi;u^*)\les J(t,\xi;u), \q (\hb{respectively,}~\h J(t,\xi;u^*)\les \h J(t,\xi;u),~\as);$$
\item ({\it uniquely}) {\it open-loop solvable at $t$} if it is (uniquely) open-loop solvable at $(t,\xi)$
      for all $\xi\in L^2_{\cF_t}(\Om;\dbR^n)$;
\item ({\it uniquely}) {\it open-loop solvable on $[0,T]$} if it is (uniquely) open-loop solvable at any $t\in[0,T]$.
\end{enumerate}
\end{definition}

One sees that Problem $\h{\rm(SLQ)}$ is stronger than Problem (SLQ) in the sense that each open-loop optimal control $u^*\in\cU[t,T]$ of Problem $\h{\rm(SLQ)}$ is also an open-loop optimal control of Problem (SLQ). Moreover, one sees that
$$ V(t,\xi)=\dbE[\h V(t,\xi)], \qq\forall (t,\xi)\in\cD. $$
Later, we will further show that if $u^*\in\cU[t,T]$ is an open-loop optimal control of Problem (SLQ),
it is also open-loop optimal for Problem $\h{\rm(SLQ)}$ (see \autoref{thm:equivalence-SLQ12}).
Therefore, these two problems are equivalent.

\ms

The study of SLQ problems was initiated by Wonham \cite{Wonham 1968} in 1968, and was later investigated by many researchers; see, for example, Athens \cite{Athens 1971}, Bismut \cite{Bismut 1976, Bismut 1978},
Davis \cite{Davis 1977}, Bensoussan \cite{Bensoussan 1982} and the references cited therein for most (if not all) major works during 1970--1980s. See also Chapter 6 of the book by Yong and Zhou \cite{Yong-Zhou 1999} for a self-contained presentation. More recent works will be briefly surveyed below.

\ms

For SLQ problems, there are three closely related objects/notions involved: (open-loop) solvability, optimality system which is a
coupled forward-backward stochastic differential equation (FBSDE, for short), and a Riccati equation. It is well known that when the map $u\mapsto J(t,\xi;u)$ is uniformly convex for every $(t,\xi)\in\cD$, which is guaranteed by the following {\it standard condition}:
\bel{classical} G\ges0,\q Q(\cd)\ges0,\q S(\cd)=0,\q R(\cd)\ges\d I_m,~ \hb{~for some~}\d>0,\ee
Problem (SLQ) is uniquely (open-loop) solvable. Then, by a variational method (or Pontryagin's maximum principle), the optimality system (a coupled FBSDE) automatically admits an adapted solution. Applying the idea of invariant imbedding \cite{Bellman-Wing 1975}, an associated Riccati equation can be formally derived, which decouples the coupled FBSDE. Now, if such a Riccati equation admits a solution, by completing squares, an (open-loop) optimal control of state feedback form can be constructed. This then solves Problem (SLQ). The same idea also applies to Problem $\h{\rm(SLQ)}$. We should point out that such a methodology, which could be called the {\it ``uniform convexity-FBSDE-Riccati equation'' approach}, for convenience, is the most natural approach to all LQ problems. For SLQ problems with deterministic coefficients (by which we mean that all the coefficients of the state equation and all the weighting matrices in the cost functional are deterministic), which includes the deterministic LQ problems, the above approach is very successful under the standard condition \rf{classical} (see Yong and Zhou \cite[Chapter 6]{Yong-Zhou 1999}).

\ms

In 1977, Molinari \cite{Molinari 1977} showed that $Q(\cd)\ges0$ is not necessary for the (open-loop) solvability of the deterministic LQ problems (see also You \cite{You 1983} for the LQ problem in Hilbert spaces), and actually, $G\ges0$ is not necessary either, although $R(\cd)\ges0$ is necessary. Furthermore, for SLQ problems, even $R(\cd)\ges0$ is not necessary for the (open-loop) solvability (see the work of Chen, Li and Zhou in 1998 \cite{Chen-Li-Zhou 1998}). Note that our assumptions (A1)--(A2) allow all the coefficients of the state equation \rf{state} and the weighting matrices in \rf{rv-L} to be stochastic processes, and no any positive/nonnegative definiteness conditions imposed on the weighting matrices $G$, $Q(\cd)$, and $R(\cd)$. Because of this, we refer to our Problems (SLQ) and $\h{\rm(SLQ)}$ as {\it indefinite SLQ problems with random coefficients}.
The indefinite SLQ problem not only stands out on its own as an interesting mathematically theoretic problem, but also has promising applications in practical areas. For example, as a special indefinite case, the matrix $R(\cd)$ is inherently zero in the mean-variance
portfolio selection problem \cite{Zhou-Li 2000,Lim-Zhou 2002};
in a pollution control model formulated in \cite{Chen-Li-Zhou 1998}, the matrix $R(\cd)$ is negative definite. The finding of \cite{Chen-Li-Zhou 1998} has triggered extensive research on the indefinite SLQ problem; see, for example, the follow-up works of Lim and Zhou \cite{Lim-Zhou 1999}, Chen and Zhou \cite{Chen-Zhou 2000}, Chen and Yong \cite{Chen-Yong 2000, Chen-Yong 2001}, Ait Rami, Moore, and Zhou \cite{Rami-Moore-Zhou 2002}, as well as the works of Hu and Zhou \cite{Hu-Zhou 2003}, and Qian and Zhou \cite{Qian-Zhou 2013}.

\ms

Without assuming any positive definiteness/semi-definiteness on the weighting matrices brings a great challenge for solving the SLQ problem. For the deterministic coefficient case, the recent results by Sun and Yong \cite{Sun-Yong 2014}, Sun, Li, and Yong \cite{Sun-Li-Yong 2016} are quite satisfactory. Let us briefly present some relevant results here. First of all, we recall the following definition (for SLQ problems with deterministic coefficients).

\begin{definition}\rm
Let $t\in[0,T)$ be a deterministic initial time, and  $L^2(t,T;\dbR^{m\times n})$ be the space of
all $\dbR^{m\times n}$-valued deterministic functions that are square-integrable on $[t,T]$.
A pair $(\Th^*,v^*)\in L^2(t,T;\dbR^{m\times n})\times\cU[t,T]$ is called a {\it closed-loop optimal
strategy} of Problem (SLQ) on $[t,T]$ if for any initial state $\xi\in L^2_{\cF_t}(\Om;\dbR^n)$ and
any $(\Th,v)\in L^2(t,T;\dbR^{m\times n})\times\cU[t,T]$,
\bel{J(Th*)} J(t,\xi;\Th^*X^*+v^*)\les J(t,\xi;\Th X+v), \ee
where $X^*=\{X^*(s);t\les s\les T\}$ is the solution to the following {\it closed-loop system}:
\begin{equation}\label{dX*}\left\{\begin{aligned}
  dX^*(s) &=\big\{[A(s)+B(s)\Th^*(s)]X^*(s) + B(s)v^*(s)\big\}ds\\
          &\hp{=\ } +\big\{[C(s)+D(s)\Th^*(s)]X^*(s) + D(s)v^*(s)\big\}dW(s),\\
   X^*(t) &=\xi,
\end{aligned}\right.\end{equation}
and $X=\{X(s);t\les s\les T\}$ on the right-hand side of \rf{J(Th*)} is the solution to \rf{dX*}
in which $(\Th^*,v^*)$ is replaced by $(\Th,v)$.
When a closed-loop optimal strategy exists on $[t,T]$, we say that Problem (SLQ) is {\it closed-loop solvable}
(on $[t,T]$).
\end{definition}

The point that we want to make here is that the closed-loop optimal strategy $(\Th^*,v^*)$ is independent of the initial state $\xi$. For open-loop and closed-loop solvabilities of Problem (SLQ) with deterministic coefficients, the following results were established in \cite{Sun-Yong 2014, Sun-Li-Yong 2016}.
\begin{itemize}
\item Problem (SLQ) is open-loop solvable at some initial pair $(t,\xi)$ if and only if the mapping $u\mapsto J(t,0;u)$ is convex and the corresponding FBSDE is solvable;
\item Problem (SLQ) is closed-loop solvable on $[t,T]$ if and only if the corresponding Riccati equation admits a regular solution;
\item If Problem (SLQ) is closed-loop solvable on $[0,T]$, then it is open-loop solvable, and every open-loop optimal control admits a closed-loop representation which must coincide with the outcome of an closed-loop optimal strategy.
      %
\end{itemize}

\ms

For the random coefficient case, we will still have the equivalence between the open-loop solvability and the solvability of a certain FBSDE (together with the convexity of the cost functional). However, the Riccati equation associated with Problem (SLQ) becomes a nonlinear BSDE, which is usually referred to as the {\it stochastic Riccati equation} (SRE, for short).
In 2003, Tang \cite{Tang 2003} and Kohlmann--Tang \cite{Kohlmann-Tang 2003} (see also \cite{Tang 2016}) proved that the associated SRE is uniquely solvable
under either the standard condition \rf{classical} or the following condition:
\begin{equation}\label{n=m}
D(\cd)^\top D(\cd)\ges\d I_m \hb{~and~} G\ges\d I_n ~\hb{for some } \d>0,
  \q Q(\cd),R(\cd)\ges0,\q S(\cd)=0,
\end{equation}
and that the corresponding closed-loop system is well-posed.
We mention that Problem (SLQ) with random coefficients under the standard condition \rf{classical} was formally posed as an open question by Bismut \cite{Bismut 1976} (see also \cite{Peng 1999}). Therefore, \cite{Tang 2003,Kohlmann-Tang 2003} can be regarded as a solution to the Bismut's open question. On the other hand, the approach used in \cite{Tang 2003,Kohlmann-Tang 2003,Tang 2016} heavily depends on the positive (semi-)definiteness assumption on the weighting matrices.

\ms

Our major concern here is the indefinite situation (with random coefficients). Hence, the problem that we are investigating can be regarded as an extended Bismut's problem. Due to the indefinite nature of our problem with random coefficients, techniques used in previous works (in particular those used in \cite{Tang 2003,Kohlmann-Tang 2003,Tang 2016}) are not (directly) applicable. Note that in the current case, the associated Riccati equation becomes a nonlinear backward stochastic differential equation (BSDE, for short) whose adapted solution $(P,\L)$ has the feature that $P$ does not have to be positive definite, and $\L$ might be unbounded in general. Consequently, even if $R+D^\top PD$ is uniformly positive definite, the process
$$\Th^*=-(R+D^\top PD)^{-1}(B^\top P+D^\top PC+D^\top\L+S)$$
(which is a closed-loop optimal strategy in the deterministic coefficient case) might be unbounded. With such a $\Th^*$, the well-posedness of the closed-loop system \rf{dX*} is not obvious because the usual uniform Lipschitz condition is not satisfied. At the moment, we feel that it is unclear whether the framework of closed-loop solvability introduced by Sun and Yong \cite{Sun-Yong 2014} (for deterministic coefficient case) can be adopted to SLQ problems with random coefficients. Therefore we will concentrate on open-loop solvability (without pursuing the closed-loop solvability) in this paper, and for simplicity of terminology, we will suppress the word ``open-loop'' in the sequel, unless it is necessarily to be emphasized.

\ms

We mention that in a recent paper by Li, Wu, and Yu \cite{Li-Wu-Yu 2018}, a very special type of indefinite SLQ problems with
random coefficients (allowing some random jumps) was studied.
The crucial assumption imposed there was that the problem admits a so-called relax compensator that transforms the indefinite problem to a problem satisfying the standard condition \rf{classical}. With such an assumption, the usual arguments apply. However, it is not clear when such a compensator exists and whether the existence of a relax compensator is necessary for the solvability of the SLQ problem.
On the other hand, a notion of feedback control was recently introduced by L\"u, Wang, and Zhang \cite{Lu-Wang-Zhang 2017} for indefinite SLQ problems with random coefficients. These feedback controls look like closed-loop strategies, but the space to which they belong is unclear.

\ms

In this paper, we shall carry out a thorough investigation on the indefinite SLQ problem with random coefficients. We will first represent the cost functional of Problem (SLQ) as a bilinear form in a suitable Hilbert space, in terms of adapted solutions of FBSDEs (A special case was presented in \cite{Chen-Yong 2001}, with a longer proof). This will be convenient from a different viewpoint. Then, similar to \cite{Mou-Yong 2006}, we will show that in order the SLQ problem to admit an optimal control, the cost functional has to be convex in the control variable; and that the uniform convexity of the cost functional (which is slightly stronger than the convexity) is a sufficient condition for the existence of a unique optimal control
(see \autoref{crlry:existence-iff}). Next, under the uniform convexity condition, we shall prove that the fundamental matrix process $\BX(\cd)$ corresponding the optimal state process is invertible (see \autoref{thm:BX-invertible}) by considering a certain stopped SLQ problem and through this, we will further establish the unique solvability of the associated SRE (see \autoref{thm:Riccati}).
With the unique solvability of the SRE, we will be able to obtain a closed-loop representation of the open-loop optimal control. It is also worth noting that the SLQ problem might still be solvable even if the cost functional is merely convex. The significance of \autoref{thm:Riccati} is that it bridges the gap between uniform convexity and convexity. In fact, by considering a perturbed SLQ problem, \autoref{thm:Riccati} makes it possible to develop an $\e$-approximation scheme that is asymptotically optimal.
This idea was first introduced by Sun, Li, and Yong \cite{Sun-Li-Yong 2016} and could be applied to the random coefficient case without any difficulties. Concerning the uniform convexity of the cost functional (in $u$), we point out that the conditions \rf{classical} and \rf{n=m} are very special cases of the uniform convexity condition we have assumed in this paper. We will present some classes of problems for which neither \rf{classical} nor \rf{n=m} holds but the cost functional is uniformly convex. Finally, we point out that considering only one-dimensional Brownian motion is just for simplicity;
multi-dimensional cases can be treated similarly without essential difficulty.

\ms

The rest of the paper is organized as follows. In \autoref{Sec:Preliminaries}, we collect some preliminary results.
\autoref{sec:Hilbert-View} is devoted to the study of the SLQ problem from a Hilbert space point of view. In \autoref{sec:equivalence}, we establish the equivalence between Problems (SLQ) and $\h{\rm(SLQ)}$.
Among other things, we present a characterization of optimal controls in terms of FSDEs. In preparation for the proof of the solvability of SREs, we investigate some basic properties of the value flow in \autoref{sec:value-process}. We discuss in \autoref{sec:SRE} the solvability of SREs, as well as the closed-loop representation of
open-loop optimal controls. Some sufficient conditions for the uniform convexity of the cost functional in $u$ will be presented in \autoref{sec:uniform-convexity}. An interesting non-trivial example will be presented in \autoref{sec:example}. Finally, some concluding remarks, including the form of the results for multi-dimensional Brownian motion case, are collected in \autoref{sec:remarks}.

\section{Preliminaries}\label{Sec:Preliminaries}

In this section we collect some preliminary results which are of frequent use in the sequel.
We begin with some notations:
\begin{align*}
\dbR^n:
&\hb{~~the $n$-dimensional Euclidean space with the Eucliden norm $|\cd|$.} \\
\dbR^{n\times m}:
&\hb{~~the Euclidean space of all $(n\times m)$ real matrices; $\dbR^n=\dbR^{n\times 1}$; $\dbR=\dbR^1$.} \\
\dbS^n:
&\hb{~~the space of all symmetric $(n\times n)$ real matrices.} \\
I_n:
&\hb{~~the identity matrix of size $n$.} \\
M^\top:
&\hb{~~the transpose of a matrix $M$.} \\
\tr(M):
&\hb{~~the trace of a matrix $M$.} \\
\lan\,\cd\,,\,\cd\ran:
&\hb{~~the Frobenius inner product on $\dbR^{n\times m}$, which is defiend by $\lan A,B\ran=\tr(A^\top B)$.} \\
|M|:
&\hb{~~the Frobenius norm of a matrix $M$, defined by $\big[\tr(A^\top B)\big]^{1\over2}$.}
\end{align*}
Recall that $\cX_t\equiv L^2_{\cF_t}(\Om;\dbR^n)$ is the space of all $\cF_t$-measurable, $\dbR^n$-valued random variables
$\xi$ with $\dbE|\xi|^2<\i$, and that $\cU[t,T]\equiv L_\dbF^2(t,T;\dbR^m)$ is the space of $\dbF$-progressively measurable,
$\dbR^m$-valued processes $u=\{u(s);t\les s\les T\}$ such that $\dbE\int_t^T|u(s)|^2ds<\i$.
To avoid prolixity later, we further introduce the following spaces of random variables and processes: For Euclidean space $\dbH=\dbR^n,\dbR^{m\times n},\dbS^n$, etc. and $p,q\ges1$,
\begin{align*}
L^\i_{\cF_t}(\Om;\dbH):
   &\hb{~~the space of bounded, $\cF_t$-measurable, $\dbH$-valued random variables.} \\
L_\dbF^q(\Om;L^p(t,T;\dbH)):&\hb{~~the space of $\dbF$-progressively measurable processes $X:[t,T]\times\Om\to\dbH$}\\
&\hb{~~with $\dbE\(\int_t^T|X(s)|^pds\)^{q\over p}<\infty$.}\\
L_\dbF^\i(\Om;L^p(t,T;\dbH)):&\hb{~~the space of $\dbF$-progressively measurable processes $X:[t,T]\times\Om\to\dbH$}\\
&\hb{~~with $\esssup_{\om\in\Om}\int_t^T|X(s,\om)|^pds<\infty$.}\\
L_\dbF^p(\Om;C([t,T];\dbH)):
   &\hb{~~the space of $\dbF$-adapted, continuous processes $X:[t,T]\times\Om\to\dbH$} \\
   &\hb{~~with $\dbE\big[\sup_{t\les s\les T}|X(s)|^p\big]<\i$.}\\
L_\dbF^\i(\Om;C([t,T];\dbH)):
   &\hb{~~the space of bounded, $\dbF$-adapted, continuous, $\dbH$-valued processes.}
\end{align*}
We denote $L^p_\dbF(\Om;L^p(t,T;\dbH))=L^p_\dbF(t,T;\dbH)$. Note that both $\cX_t$ and $\cU[t,T]$ are Hilbert spaces under their natural inner products. We shall use
$$ [\![ u,v]\!] = \dbE\int_t^T\lan u(s),v(s)\ran ds,$$
to denote the inner product of $u,v\in\cU[t,T]$, distinguishing it from the Euclidean inner product on a Euclidean space.
\ss

Next we recall some results concerning existence and uniqueness of solutions to forward SDEs
(FSDEs, for short) and BSDEs with random coefficients. Consider the linear FSDE
\bel{SDE0}\left\{\begin{aligned}
   dX(s) &= [A(s)X(s)+b(s)]ds + [C(s)X(s)+\si(s)]dW(s),\q s\in[t,T],\\
    X(t) &= \xi,
\end{aligned}\right.\ee
and the linear BSDE
\bel{BSDE0}\left\{\begin{aligned}
   dY(s) &=-[A(s)^\top Y(s)+C(s)^\top Z(s)+\f(s)]ds + Z(s)dW(s), \q s\in[t,T],\\
    Y(T) &= \eta.
\end{aligned}\right.\ee
We have the following result.

\begin{lemma}\label{lmm:sol-F&B-SDE} \sl
Suppose that
$$A(\cd)\in L^\i_\dbF(\Om;L^1(0,T;\dbR^{n\times n})),\qq C(\cd)\in L^\i_\dbF(\Om;L^2(0,T;\dbR^{n\times n})).$$
Then the following hold:
\begin{enumerate}[\q\rm(i)]
\item For any initial pair $(t,\xi)\in\cD$ and any processes $b\in L^2_\dbF(\Om;L^1(t,T;\dbR^n))$, $\si\in L_\dbF^2(t,T;\dbR^n)$,
      \rf{SDE0} has a unique solution $X$, which belongs to the space $L_\dbF^2(\Om;C([t,T];\dbR^n))$.
\item For any terminal state $\eta\in L^2_{\cF_T}(\Om;\dbR^n)$ and any $\f\in L^2_\dbF(\Om;L^1(t,T;\dbR^n))$,
      \rf{BSDE0} has a unique adapted solution $(Y,Z)$, which belongs to the space
      $L_\dbF^2(\Om;C([t,T];\dbR^n))\times L^2_\dbF(t,T;\dbR^n)$.
\end{enumerate}
Moreover, there exists a constant $K>0$, depending only on $A$, $C$, and $T$, such that
\begin{align*}
  &\dbE\[\sup_{t\les s\les T}|X(s)|^2\]
       \les K\dbE\[|\xi|^2+\(\int_t^T|b(s)|ds\)^2+\int_t^T|\si(s)|^2ds\],\\
  &\dbE\[\sup_{t\les s\les T}|Y(s)|^2 + \int_t^T|Z(s)|^2ds\]
       \les K\dbE\[|\eta|^2+\(\int_t^T|\f(s)|ds\)^2\].
\end{align*}
\end{lemma}

Note that in \autoref{lmm:sol-F&B-SDE}, the coefficients $A$ and $C$ are allowed to be unbounded,
which is a little different from the standard case.
However, the proof of \autoref{lmm:sol-F&B-SDE} is almost the same as that of
\cite[Proposition 2.1]{Sun-Yong 2014}.
So we omit the details here and refer the reader to \cite{Sun-Yong 2014}.

\ms

Consider now the following BSDE for $\dbS^n$-valued processes over the interval $[0,T]$:
\bel{BSDE-MN}\left\{\begin{aligned}
   dM(s) &=-\big[M(s)A(s) + A(s)^\top M(s) + C(s)^\top M(s)C(s) \\
   &\hp{=-\big[} +N(s)C(s) + C(s)^\top N(s)+ Q(s)\big]ds + N(s)dW(s), \\
    M(T) &=G.
\end{aligned}\right.\ee
From \autoref{lmm:sol-F&B-SDE} (ii) it follows that under the assumptions \ref{(A1)}--\ref{(A2)},
equation \rf{BSDE-MN} admits a unique square-integrable adapted solution $(M,N)$.
The following result further shows that $M=\{M(s);0\les s\les T\}$ is actually a bounded process.

\begin{proposition}\label{prop:M-bounded} \sl
Let {\rm\ref{(A1)}--\ref{(A2)}} hold. Then the first component $M$ of the adapted solution
$(M,N)$ to the BSDE \rf{BSDE-MN} is bounded.
\end{proposition}

\begin{proof}
Let $\b>0$ be undetermined and denote
\bel{Pi-7-12} \Pi=MA + A^\top M + C^\top MC +NC + C^\top N+ Q. \ee
Note that we have suppressed the argument $s$ in \rf{Pi-7-12} and will do so hereafter whenever
there is no confusion. Applying It\^o's formula to $s\mapsto e^{\b s}|M(s)|^2$ yields
\begin{align}\label{estimate-M}
e^{\b t}|M(t)|^2
   &= e^{\b T}|G|^2 +\int_t^Te^{\b s}\[-\b|M(s)|^2 +2\lan M(s),\Pi(s)\ran -|N(s)|^2\]ds \nn\\
   &\hp{=\ } -2\int_t^Te^{\b s}\lan M(s),N(s)\ran dW(s), \q\forall 0\les t\les T,~\as
\end{align}
By \ref{(A1)}--\ref{(A2)}, the processes $A$, $C$, and $Q$ are bounded.
Thus, there exists a constant $K>0$ such that
$$|\Pi(s)|\les K\big[|M(s)|+|N(s)|+1\big], \q \ae~s\in[0,T],~\as$$
Using the Cauchy-Schwarz inequality, one has
\begin{align*}
2\lan M(s),\Pi(s)\ran &\les 2K|M(s)|\cd[|M(s)|+|N(s)|+1] \\
                         &= 2K|M(s)|^2+2K|M(s)|\cd|N(s)|+2K \\
                      &\les 2K|M(s)|^2+K^2|M(s)|^2 +|N(s)|^2+2K \\
                         &= (K^2+2K)|M(s)|^2 +|N(s)|^2+2K.
\end{align*}
Substituting this estimate back into \rf{estimate-M} and then taking $\b=K^2+2K$, we obtain
$$ e^{\b t}|M(t)|^2
   \les e^{\b T}|G|^2 +\int_t^T 2Ke^{\b s}ds -2\int_t^Te^{\b s}\lan M(s),N(s)\ran dW(s). $$
Observing that $\int_0^\cd e^{\b s}\lan M(s),N(s)\ran dW(s)$ is a martingale, we may take conditional
expectations with respect to $\cF_t$ on both sides of the above to obtain
$$ |M(t)|^2 \les e^{\b t}|M(t)|^2
            \les e^{\b T}\dbE\[|G|^2\,\big|\,\cF_t\] +\int_0^T 2Ke^{\b s}ds, \q\forall t\in[0,T]. $$
The assertion follows, since $G\in L^\i_{\cF_T}(\Om;\dbS^n)$.
\end{proof}

\section{A Hilbert Space Point of View}\label{sec:Hilbert-View}

Inspired by \cite{Mou-Yong 2006}, we study in this section the SLQ problem from a Hilbert space point of view.
Following the idea of \cite{Chen-Yong 2001}, we shall derive a functional representation of $J(t,\xi;u)$,
which has several important consequences and plays a basic role for the analysis of the stochastic value
flow $\h V(t,\xi)$ in \autoref{sec:value-process}.
As mentioned earlier, for notational convenience we will frequently suppress the $s$-dependence of a stochastic
process when it is involved in a differential equation or an integral.

\ms

First, we present a simple lemma.

\begin{lemma}\label{lmm:J-rep} \sl
Let {\rm\ref{(A1)}--\ref{(A2)}} hold.
Then for any initial pair $(t,\xi)\in\cD$ and control $u\in\cU[t,T]$,
\begin{equation}\label{hatJ-biaoshi1}
\h J(t,\xi;u) = \lan Y(t),\xi\ran
     +\dbE\[\int_t^T\llan B^\top Y+D^\top Z+SX+Ru,u\rran ds\bigm|\cF_t\],
\end{equation}
where $(X,Y,Z)$ is the adapted solution to the following controlled decoupled linear FBSDE:
\bel{FBSDE-XYZu}\left\{\begin{aligned}
   & dX(s)=(AX+Bu)ds + (CX+Du)dW(s), \\
   & dY(s)=-\big(A^\top Y+C^\top Z+QX+S^\top u\big)ds+ZdW(s), \\
   & X(t)=\xi, \q Y(T)=GX(T).
\end{aligned}\right.\ee
\end{lemma}

\begin{proof}
Note that the FSDE in \rf{FBSDE-XYZu} is exactly the state equation \rf{state}.
Applying It\^{o}'s formula to $s\mapsto\lan Y(s),X(s)\ran$ yields
\begin{align}\label{GX(T)X(T)}
\lan GX(T),X(T)\ran
  &=\lan Y(t),\xi\ran + \int_t^T\[\blan B^\top Y+D^\top Z-SX,u\bran-\lan QX,X\ran\]ds \nn\\
  &\hp{=\ }\qq +\int_t^T\[\lan Z,X\ran+\lan Y,CX+Du\ran\]dW(s).
\end{align}
Substituting \rf{GX(T)X(T)} into $\h J(t,\xi;u)$ and noting that
$$\dbE\[\int_t^T\(\lan Z,X\ran+\lan Y,CX+Du\ran\)dW(s)\bigm|\cF_t\]=0,$$
we obtain \rf{hatJ-biaoshi1}.
\end{proof}

The adapted solution $(X,Y,Z)$ to the FBSDE \rf{FBSDE-XYZu} is determined jointly by the initial state $\xi$
and the control $u$. To separate $\xi$ and $u$, let $(\wt X,\wt Y,\wt Z)$ and $(\bar X,\bar Y,\bar Z)$ be the
adapted solutions to the decoupled linear FBSDEs
\bel{FBSDE-tiXYZu}\left\{\begin{aligned}
   & d\wt X(s)=(A\wt X+Bu)ds + (C\wt X+Du)dW(s), \\
   & d\wt Y(s)=-\big(A^\top\wt Y+C^\top\wt Z+Q\wt X+S^\top u\big)ds + \wt ZdW(s), \\
   & \wt X(t)=0, \q \wt Y(T)=G\wt X(T),
\end{aligned}\right.\ee
and
\bel{FBSDE-barXYZ}\left\{\begin{aligned}
   & d\bar X(s)=A\bar Xds+C\bar XdW(s), \\
   & d\bar Y(s)=-\big(A^\top\bar Y+C^\top\bar Z+Q\bar X\big)ds +\bar Z dW(s), \\
   & \bar X(t)=\xi, \q\bar Y(T)=G\bar X(T),
\end{aligned}\right.\ee
respectively. Then $(X,Y,Z)$ can be written as the sum of $(\wt X,\wt Y,\wt Z)$ and $(\bar X,\bar Y,\bar Z)$:
$$X(s) = \wt X(s)+\bar X(s), \q Y(s) = \wt Y(s)+\bar Y(s), \q Z(s) =\wt Z(s)+\bar Z(s); \q s\in[t,T].$$
Note that $(\wt X,\wt Y,\wt Z)$ (respectively, $(\bar X,\bar Y,\bar Z)$) depends linearly on $u$ (respectively, $\xi$) alone.
We now define two linear operators
$$\cN_t:\cU[t,T]\to\cU[t,T], \q \cL_t:\cX_t\to\cU[t,T]$$
as follows: For any $u\in\cU[t,T]$, $\cN_t u$ is defined by
\begin{equation}\label{cN-def}
   [\cN_t u](s) = B(s)^\top\wt Y(s)+D(s)^\top\wt Z(s)+S(s)\wt X(s)+R(s)u(s), \q s\in[t,T],
\end{equation}
and for any $\xi\in\cX_t$, $\cL_t\xi$ is defined by
\begin{equation}\label{cL-def}
   [\cL_t\xi](s) = B(s)^\top\bar Y(s)+D(s)^\top \bar Z(s)+S(s)\bar X(s), \q s\in[t,T].
\end{equation}
For these two operators, we have the following result.

\begin{proposition} \sl
Let {\rm\ref{(A1)}--\ref{(A2)}} hold. Then
\begin{enumerate}[\q\rm(i)]
\item the linear operator $\cN_t$ defined by \rf{cN-def} is a bounded self-adjoint operator
      on the Hilbert space $\cU[t,T]$;
\item the linear operator $\cL_t$ defined by \rf{cL-def} is a bounded operator from the Hilbert
      space $\cX_t$ into the Hilbert space $\cU[t,T]$.
      Moreover, there exists a constant $K>0$ independent of $t$ and $\xi$ such that
      \bel{uni-bound-Lt}[\![ \cL_t\xi,\cL_t\xi ]\!] \les K\,\dbE|\xi|^2, \q\forall \xi\in\cX_t.\ee
\end{enumerate}
\end{proposition}

\begin{proof}
(i) The boundedness of $\cN_t$ is a direct consequence of the estimates in \autoref{lmm:sol-F&B-SDE}.
To prove that $\cN_t$ is self-adjoint, it suffices to show that for any $u_1,u_2\in\cU[t,T]$,
\bel{self-adjoint}\dbE\int_t^T\blan[\cN_tu_1](s),u_2(s)\bran ds
= \dbE\int_t^T\blan u_1(s),[\cN_tu_2](s)\bran ds. \ee
To this end, we take two arbitrary processes $u_1,u_2\in\cU[t,T]$
and let $(\wt X_i,\wt Y_i,\wt Z_i)$ $(i=1,2)$ be the adapted solution to \rf{FBSDE-tiXYZu} in which
$u$ is replaced by $u_i$.
Applying It\^{o}'s formula to $s\mapsto\blan\wt Y_2(s),\wt X_1(s)\bran$ yields
$$\dbE\blan G\wt X_2(T),\wt X_1(T)\bran
= \dbE\int_t^T\[\blan B^\top\wt Y_2+D^\top\wt Z_2,u_1\bran -\blan Q\wt X_2,\wt X_1\bran -\blan S\wt X_1,u_2\bran\]ds, $$
and applying It\^{o}'s formula to $s\mapsto\blan\wt Y_1(s),\wt X_2(s)\bran$ yields
$$\dbE\blan G\wt X_1(T),\wt X_2(T)\bran
= \dbE\int_t^T\[\blan B^\top\wt Y_1+D^\top\wt Z_1,u_2\bran -\blan Q\wt X_1,\wt X_2\bran -\blan S\wt X_2,u_1\bran\]ds. $$
Combining the above two equations and noting that $G$ and $Q$ are symmetric, we obtain
\bel{add-1} \dbE\int_t^T\blan B^\top\wt Y_1+D^\top\wt Z_1+S\wt X_1,u_2\bran ds
= \dbE\int_t^T\blan B^\top\wt Y_2+D^\top\wt Z_2+S\wt X_2,u_1\bran ds. \ee
Note that because $R$ is symmetric,
\bel{add-2} \dbE\int_t^T\lan Ru_1,u_2\ran ds = \dbE\int_t^T\lan Ru_2,u_1\ran ds, \ee
and that by the definition of $\cN_t$,
$$[\cN_t u_i](s) = B(s)^\top\wt Y_i(s) +D(s)^\top\wt Z_i(s) +S(s)\wt X_i(s) +R(s)u_i(s), \q s\in[t,T].$$
Adding \rf{add-2} to \rf{add-1} gives \rf{self-adjoint}.

\ms

(ii) It suffices to prove \rf{uni-bound-Lt}. Choose a constant $\a>0$ such that
\bel{alpha}|G|^2,|B(s)|^2,|D(s)|^2,|S(s)|^2,|Q(s)|^2\les\a, \q\ae~s\in[0,T],~\as \ee
Then by using the inequality $|v_1+\cds+v_k|^2\les k(|v_1|^2+\cds+|v_k|^2)$, we obtain
\begin{align*}
[\![ \cL_t\xi,\cL_t\xi ]\!]
   &=\dbE\int_t^T |B(s)^\top\bar Y(s)+D(s)^\top \bar Z(s)+S(s)\bar X(s)|^2 ds \\
   &\les 3\a\,\dbE\int_t^T \[|\bar Y(s)|^2+|\bar Z(s)|^2+|\bar X(s)|^2\] ds.
\end{align*}
By \autoref{lmm:sol-F&B-SDE}, there exists a constant $\b>0$, independent of $t$ and $\xi$, such that
\begin{eqnarray}
\label{5-17-1}
&\ds \dbE\int_t^T \[|\bar Y(s)|^2+|\bar Z(s)|^2\] ds  \les \b\,\dbE\[|G\bar X(T)|^2 +\int_t^T|Q(s)\bar X(s)|^2ds\], \\
\label{5-17-2}
&\ds \dbE|\bar X(T)|^2 +\dbE\int_t^T |\bar X(s)|^2 ds \les \b\,\dbE|\xi|^2.
\end{eqnarray}
Substituting \rf{5-17-2} into \rf{5-17-1} and making use of \rf{alpha}, we further obtain
$$\dbE\int_t^T \[|\bar Y(s)|^2+|\bar Z(s)|^2\] ds  \les \a\b^2\,\dbE|\xi|^2.$$
It follows that $[\![ \cL_t\xi,\cL_t\xi ]\!] \les 3\a(\a\b^2+\b)\,\dbE|\xi|^2$ for all $\xi\in\cX_t$. \end{proof}

\begin{remark}\rm
Let $(\dbX,\dbY,\dbZ)$ be the adapted solution to the decoupled linear FBSDE
for $\dbR^{n\times n}$-valued processes:
$$\left\{\begin{aligned}
   & d\dbX(s)=A\dbX ds+C\dbX dW(s), \\
   & d\dbY(s)=-\big(A^\top\dbY+C^\top\dbZ+Q\dbX\big)ds +\dbZ dW(s), \\
   & \dbX(0)=I_n, \q\dbY(T)=G\dbX(T).
\end{aligned}\right.$$
It is straightforward to verify that $\dbX$ has an inverse $\dbX^{-1}$ which satisfies
$$\left\{\begin{aligned}
   d\dbX^{-1}(s) &= \dbX^{-1}(C^2-A)ds -\dbX^{-1}CdW(s), \q s\in[0,T],\\
    \dbX^{-1}(0) &=I_n.
\end{aligned}\right.$$
Observe that for any $\xi\in L^\i_{\cF_t}(\Om;\dbR^n)$, the processes
$$\dbX(s)\dbX^{-1}(t)\xi, \q \dbY(s)\dbX^{-1}(t)\xi, \q \dbZ(s)\dbX^{-1}(t)\xi; \q s\in[t,T],$$
are all square-integrable and satisfy the FBSDE \rf{FBSDE-barXYZ}.
Hence, by uniqueness of adapted solutions, we must have
$$(\bar X(s),\bar Y(s),\bar Z(s)) =
(\dbX(s)\dbX^{-1}(t)\xi, \dbY(s)\dbX^{-1}(t)\xi, \dbZ(s)\dbX^{-1}(t)\xi); \q s\in[t,T].$$
Therefore, if $\xi\in L^\i_{\cF_t}(\Om;\dbR^n)$, then $\cL_t\xi$ can be represented, in terms of $(\dbX,\dbY,\dbZ)$, as
\bel{Lt-rep*}[\cL_t\xi](s) = [B(s)^\top\dbY(s)+D(s)^\top\dbZ(s)+S(s)\dbX(s)]\dbX^{-1}(t)\xi, \q s\in[t,T].\ee
This relation will be used in \autoref{sec:value-process}.
\end{remark}

We are now ready to present the functional representation of the cost functional $J(t,\xi;u)$.
Observe that $J(t,\xi;u)$ and $\h J(t,\xi;u)$ have the relation $J(t,\xi;u)=\dbE\h J(t,\xi;u)$,
and recall that the first component $M$ of the adapted solution $(M,N)$ to the BSDE \rf{BSDE-MN}
is bounded (\autoref{prop:M-bounded}).

\begin{theorem}\label{thm:rep-cost} \sl
Let {\rm\ref{(A1)}--\ref{(A2)}} hold. Then the cost functional $J(t,\xi;u)$ admits the following representation:
\begin{equation}\label{J-rep}
  J(t,\xi;u) = [\![ \cN_tu,u ]\!] + 2[\![ \cL_t\xi,u ]\!] + \dbE\lan M(t)\xi,\xi\ran, \q\forall(t,\xi)\in\cD,
\end{equation}
where $\cN_t$, $\cL_t$ are defined by \rf{cN-def} and \rf{cL-def}, respectively, and $(M,N)$ is the adapted solution of BSDE \rf{BSDE-MN}.
\end{theorem}

\begin{proof}
Fix any $(t,\xi)\in\cD$ and $u\in\cU[t,T]$. Let $(X,Y,Z)$, $(\wt X,\wt Y,\wt Z)$, and $(\bar X,\bar Y,\bar Z)$
be the adapted solutions to \rf{FBSDE-XYZu}, \rf{FBSDE-tiXYZu}, and \rf{FBSDE-barXYZ}, respectively. Then
$$X(s) = \wt X(s)+\bar X(s), \q Y(s) = \wt Y(s)+\bar Y(s), \q Z(s) =\wt Z(s)+\bar Z(s); \q s\in[t,T].$$
By \autoref{lmm:J-rep}, the relation $J(t,\xi;u)=\dbE\h J(t,\xi;u)$,
and the definitions of $\cN_t$ and $\cL_t$, we have
\bel{J-1/4} J(t,\xi;u) = \dbE\[\blan\wt Y(t),\xi\bran + \blan\bar Y(t),\xi\bran
                         +\int_t^T\blan [\cN_tu](s) + [\cL_t\xi](s), u(s)\bran ds\]. \ee
Now applying It\^{o}'s formula to $s\mapsto\lan\wt Y(s),\bar X(s)\ran$ gives
$$\dbE\blan G\wt X(T),\bar X(T)\bran - \dbE\blan\wt Y(t),\xi\bran
= -\dbE\int_t^T\[\blan Q(s)\wt X(s),\bar X(s)\bran + \blan S(s)\bar X(s),u(s)\bran\] ds,$$
and applying It\^{o}'s formula to $s\mapsto\lan\bar Y(s),\wt X(s)\ran$ gives
$$\dbE\blan G\bar X(T),\wt X(T)\bran
= \dbE\int_t^T\[\blan B(s)^\top\bar Y(s)+D(s)^\top\bar Z(s),u(s)\bran - \blan Q(s)\bar X(s),\wt X(s)\bran\] ds.$$
Combining the last two equations we obtain
\begin{align}\label{tiY(t)xi}
\dbE\blan\wt Y(t),\xi\bran
   &= \dbE\int_t^T\blan B(s)^\top\bar Y(s)+D(s)^\top\bar Z(s)+S(s)\bar X(s),u(s)\bran ds= \dbE\int_t^T\blan [\cL_t\xi](s),u(s)\bran ds.
\end{align}
On the other hand, since $(M,N)$ is the adapted solution of \rf{BSDE-MN}, by It\^{o}'s formula, we have
$$\ba{ll}
\ns\ds d(M\bar X)=\big[-(MA+A^\top M+C^\top MC+NC+C^\top N+Q)\bar X+MA\bar X+NC\bar X\big]ds\\
\ns\ds\qq\qq\qq+\big[N\bar X+MC\bar X\big]dW(s)\\
\ns\ds\qq\q~=-\big[A^\top M\bar X + C^\top(MC+N)\bar X +Q\bar X\big]ds + (MC+N)\bar XdW(s).\ea $$
Noting $M(T)\bar X(T)=\bar Y(T)$, we see that the pair of processes $(M\bar X,(MC+N)\bar X)$
satisfies the same BSDE as $(\bar Y,\bar Z)$. Thus, by the uniqueness of adapted solutions,
$$\bar Y(s)=M(s)\bar X(s), \q \bar Z(s)=[M(s)C(s)+N(s)]\bar X(s); \q s\in[t,T].$$
It follows that $\dbE\blan\bar Y(t),\xi\bran=\dbE\lan M(t)\xi,\xi\ran$.
Substituting this and \rf{tiY(t)xi} into \rf{J-1/4} results in \rf{J-rep}.
\end{proof}

We have the following corollary to \autoref{thm:rep-cost}. A similar result can be found in \cite{Mou-Yong 2006}.

\begin{corollary}\label{crlry:existence-iff} \sl
Let {\rm\ref{(A1)}--\ref{(A2)}} hold. Let $t$ be an $\dbF$-stopping time with values in $[0,T)$.
\begin{enumerate}[\q\rm(i)]
\item A control $u^*\in\cU[t,T]$ is optimal for Problem {\rm(SLQ)} at $(t,\xi)\in\cD$ if and only if
      \begin{equation}\label{N>0-Nu+Lxi=0}
         \cN_t\ges0, \q\hb{and}\q \cN_tu^*+\cL_t\xi=0.
      \end{equation}
\item If $\cN_t$ is invertible in addition to $\cN_t\ges0$,
      then Problem {\rm(SLQ)} is uniquely solvable at $t$, and the unique optimal control
      $u^*_{t,\xi}$ at $(t,\xi)\in\cD$ is given by
      $$ u^*_{t,\xi} = -\cN_t^{-1}\cL_t\xi. $$
Consequently,
     \begin{equation}\label{V}
         V(t,\xi)=\dbE\lan[M(t)-\cL_t^*\cN_t^{-1}\cL_t]\xi,\xi\ran.
      \end{equation}
\end{enumerate}
\end{corollary}

\begin{proof}
(i) By \autoref{def:ol-opt-cntrl}, $u^*$ is optimal for Problem (SLQ) at $(t,\xi)$ if and only if
\begin{equation}\label{J(u+v)-J(u)}
   J(t,\xi;u^*+\l v)-J(t,\xi;u^*) \ges 0, \q\forall v\in\cU[t,T],~\forall \l\in\dbR.
\end{equation}
According to the representation \rf{J-rep},
\begin{align*}
J(t,\xi;u^*+\l v)
   &= [\![\cN_t(u^*+\l v),u^*+\l v]\!] + 2[\![\cL_t\xi,u^*+\l v]\!] + \dbE\lan M(t)\xi,\xi\ran \\
   &= [\![\cN_tu^*,u^*]\!] +2\l[\![\cN_tu^*,v]\!] + \l^2[\![\cN_tv,v]\!] +2[\![\cL_t\xi,u^*]\!] +2\l[\![\cL_t\xi,v]\!] + \dbE\lan M(t)\xi,\xi\ran \\
   &= J(t,\xi;u^*) + \l^2[\![\cN_tv,v]\!] + 2\l[\![\cN_tu^*+\cL_t\xi,v]\!],
\end{align*}
from which we see that \rf{J(u+v)-J(u)} is equivalent to
$$\l^2[\![\cN_tv,v]\!] + 2\l[\![\cN_tu^*+\cL_t\xi,v]\!] \ges 0,\q\forall v\in\cU[t,T],\qq\forall \l\in\dbR.$$
This means that for any arbitrarily fixed $v\in\cU[t,T]$, the quadratic function
$$f(\l)\deq \l^2[\![\cN_tv,v]\!] + 2\l[\![\cN_tu^*+\cL_t\xi,v]\!]$$
is nonnegative. So we must have
$$[\![\cN_tv,v]\!] \ges 0, \q [\![\cN_tu^*+\cL_t\xi,v]\!]=0,\qq\forall v
\in\cU[t,T],$$
leading to \rf{N>0-Nu+Lxi=0}. The converse assertion is obvious.

\ms

(ii) This is a direct consequence of (i).
\end{proof}

\section{Equivalence between Problems (SLQ) and $\h{\bf(SLQ)}$}\label{sec:equivalence}

The objective of this section is to establish the equivalence between Problems (SLQ) and $\h{\rm(SLQ)}$.
First, we present an alternative version of \autoref{crlry:existence-iff} (i),
which characterizes the solvability of Problem (SLQ) in terms of FBSDEs.

\begin{theorem}\label{thm:SLQ-open-kehua} \sl Let {\rm\ref{(A1)}--\ref{(A2)}} hold, and the initial pair $(t,\xi)\in\cD$ be given. A process $u^*\in\cU[t,T]$ is an optimal control of Problem {\rm(SLQ)} at $(t,\xi)$ if and only if the following two conditions hold:
\begin{enumerate}[\q\rm(i)]
\item the mapping $u\mapsto J(t,0;u)$ is convex, or equivalently,
      $$ J(t,0;u)\ges 0, \q\forall u\in\cU[t,T];$$
\item the adapted solution $(X,Y,Z)$ to the decoupled FBSDE
      \bel{SLQ:FBSDE}\left\{\begin{aligned}
         & dX(s)=\big[A(s)X(s)+B(s)u^*(s)\big]ds+\big[C(s)X(s)+D(s)u^*(s)\big]dW(s), \\
         & dY(s)=-\big[A(s)^\top Y(s)+C(s)^\top Z(s)+Q(s)X(s)+S(s)^\top u^*(s)\big]ds+Z(s)dW(s), \\
         & X(t)=\xi, \q Y(T)=GX(T)
      \end{aligned}\right.\ee
      satisfies the following stationarity condition:
      \begin{align}\label{stationarity}\begin{aligned}
         B(s)^\top Y(s)+D(s)^\top Z(s)+S(s)X(s)+R(s)u^*(s)=0, \q\ae~s\in[t,T],~\as
      \end{aligned}\end{align}
\end{enumerate}
\end{theorem}

\begin{proof}
By \autoref{crlry:existence-iff} (i), $u^*\in\cU[t,T]$ is an optimal control of Problem (SLQ) at $(t,\xi)$ if and only if \rf{N>0-Nu+Lxi=0} holds.
According to the representation \rf{J-rep}, $\cN_t\ges0$ is equivalent to
$$J(t,0;u)= [\![ \cN_tu,u ]\!] \ges0, \q\forall u\in\cU[t,T],$$
which is exactly the condition (i). By the definitions of $\cN_t$ and $\cL_t$, it is easy to see that
$$[\cN_tu^*+\cL_t\xi](s) = B(s)^\top Y(s)+D(s)^\top Z(s)+S(s)X(s)+R(s)u^*(s); \q s\in[t,T], $$
where $(X,Y,Z)$ is the adapted solution to the FBSDE \rf{SLQ:FBSDE}.
Thus, $\cN_tu^*+\cL_t\xi=0$ is equivalent to the condition (ii).
\end{proof}

The next result establishes the equivalence between Problems (SLQ) and $\h{\rm(SLQ)}$.

\begin{theorem}\label{thm:equivalence-SLQ12} \sl Let {\rm\ref{(A1)}--\ref{(A2)}} hold. For any given initial pair $(t,\xi)\in\cD$, a control $u^*\in\cU[t,T]$ is optimal for Problem {\rm(SLQ)} at $(t,\xi)$ if and only if it is optimal for Problem $\h{\rm(SLQ)}$ at $(t,\xi)$.
\end{theorem}

\begin{proof}
The sufficiency is trivially true. Now suppose that $u^*\in\cU[t,T]$ is optimal for Problem (SLQ) at $(t,\xi)$, and let $(X,Y,Z)$ be the adapted solution to the FBSDE \rf{SLQ:FBSDE}. To prove that $u^*$ is also
optimal for Problem $\h{\rm(SLQ)}$ at $(t,\xi)$, it suffices to show that for any set $\G\in\cF_t$,
\bel{18-7-19} \dbE[L(t,\xi;u^*){\bf1}_\G] \les \dbE[L(t,\xi;u){\bf1}_\G], \qq\forall u\in\cU[t,T]. \ee
For this, let us fix an arbitrary set $\G\in\cF_t$ and an arbitrary control $u\in\cU[t,T]$. Define
$$ \hat\xi(\om)=\xi(\om){\bf1}_\G(\om),
\q \hat u(s,\om)=u(s,\om){\bf1}_\G(\om),
\q \hat u^*(s,\om)=u^*(s,\om){\bf1}_\G(\om),$$
and consider the following FBSDE:
\begin{equation}\label{FBSDE:hatXYZ}\left\{\begin{aligned}
   & d\h X(s)=(A\h X+B\hat u^*)ds + (C\h X+D\hat u^*)dW(s), \\
   & d\h Y(s)=-\big(A^\top\h Y+C^\top\h Z+Q\h X+S^\top\hat u^*\big)ds +\h ZdW(s), \\
   & \h X(t)=\hat\xi, \q \h Y(T)=G\h X(T).
\end{aligned}\right.\end{equation}
It is straightforward to verify that the adapted solution $(\h X,\h Y,\h Z)$ of \rf{FBSDE:hatXYZ} is given by
$$ \h X(s,\om)=X(s,\om){\bf1}_\G(\om),
\q \h Y(s,\om)=Y(s,\om){\bf1}_\G(\om),
\q \h Z(s,\om)=Z(s,\om){\bf1}_\G(\om).$$
Since by \autoref{thm:SLQ-open-kehua}, $(X,Y,Z)$ satisfies the condition \rf{stationarity},
we obtain, by multiplying both sides of \rf{stationarity} by ${\bf1}_\G$, that
$$ B(s)^\top\h Y(s)+D(s)^\top\h Z(s)+S(s)\h X(s)+R(s)\hat u^*(s)=0, \q \ae~s\in[t,T],~\as $$
Applying \autoref{thm:SLQ-open-kehua}, we conclude that $\hat u^*$
is an optimal control of Problem (SLQ) at $(t,\hat\xi)$. Hence,
$$ \dbE[L(t,\hat\xi;\hat u^*)] \les \dbE[L(t,\hat\xi;\hat u)]. $$
Note that the state process $X(\cd)=X(\cd\,;t,\xi,u^*)$ corresponding to $(\xi,u^*)$ and the state process
$\h X(\cd)=X(\cd\,;t,\hat\xi,\hat u^*)$ corresponding to $(\hat\xi,\hat u^*)$ are related by
$$ X(\cd\,;t,\xi,u^*){\bf1}_\G = X(\cd\,;t,\hat\xi,\hat u^*).$$
It follows that $L(t,\xi;u^*){\bf1}_\G = L(t,\hat\xi,\hat u^*)$.
Similarly, we have $L(t,\xi;u){\bf1}_\G = L(t,\hat\xi,\hat u)$. Thus,
$$\dbE[L(t,\xi;u^*){\bf1}_\G]  =  \dbE[L(t,\hat\xi;\hat u^*)]
                             \les \dbE[L(t,\hat\xi;\hat u)]
                               =  \dbE[L(t,\xi;u){\bf1}_\G]. $$
This proves \rf{18-7-19} and therefore completes the proof.
\end{proof}

\begin{remark}\rm
We have seen from \autoref{thm:equivalence-SLQ12} that Problems (SLQ) and $\h{\rm(SLQ)}$ are equivalent.
So from now on, we will simply call both of them Problem (SLQ), although we will still have the stochastic
value flow $\h V(\cd\,,\cd)$ and the value flow $V(\cd\,,\cd)$.
\end{remark}

To conclude this section, we present some useful consequences of \autoref{thm:SLQ-open-kehua}.

\begin{corollary}\label{crlry:V=Y(t)X(t)} \sl
Let {\rm\ref{(A1)}--\ref{(A2)}} hold. Suppose that $(X^*,u^*)=\{(X^*(s),u^*(s));t\les s\les T\}$ is an
optimal pair corresponding to $(t,\xi)\in\cD$, and let $(Y^*,Z^*)=\{(Y^*(s),Z^*(s));t\les s\les T\}$ be
the adapted solution of the adjoint BSDE
$$\left\{\begin{aligned}
   dY^*(s) &= -\big(A^\top Y^*+C^\top Z^*+QX^*+S^\top u^*\big)ds+Z^*dW(s), \q s\in[t,T], \\
    Y^*(T) &= GX^*(T)
\end{aligned}\right.$$
associated with $(X^*,u^*)$. Then
$$\h V(t,\xi)=\h J(t,\xi;u^*) = \lan Y^*(t),\xi\ran. $$
\end{corollary}

\begin{proof}
Since $(X^*,u^*)$ is an optimal pair corresponding to $(t,\xi)$, we have by \autoref{thm:SLQ-open-kehua} that
$$ B(s)^\top Y^*(s)+D(s)^\top Z^*(s)+S(s)X^*(s)+R(s)u^*(s)=0, \q\ae~s\in[t,T],~\as $$
Then it follows immediately from \autoref{lmm:J-rep} that $\h V(t,\xi)=\h J(t,\xi;u^*)=\lan Y^*(t),\xi\ran$.
\end{proof}

\begin{corollary}[Principle of Optimality]\label{crlry:Bellman-Principle} \sl
Let {\rm\ref{(A1)}--\ref{(A2)}} hold. Suppose that $u^*\in\cU[t,T]$ is an optimal control at $(t,\xi)\in\cD$, and let $X^*=\{X^*(s);t\les s\les T\}$
be the corresponding optimal state process. Then for any stopping time $\t$ with $t<\t<T$, the restriction
$$u^*|_{[\t,T]} = \{u^*(s);\t\les s\les T\}$$
of $u^*$ to $[\t,T]$ is optimal at $(\t,X^*(\t))$.
\end{corollary}

The above property is called the {\it time-consistency} of the optimal control.

\begin{proof}
Let $\t$ be an arbitrary stopping time with values in $(t,T)$. According to \autoref{thm:SLQ-open-kehua},
it suffices to show that
\begin{enumerate}[\indent(a)]
\item $J(\t,0;v)\ges 0$ for all $v\in\cU[\t,T]$, and
\item the adapted solution $(X,Y,Z)$ of the decoupled FBSDE
      $$\left\{\begin{aligned}
         & dX(s)=(AX+Bu^*|_{[\t,T]})ds + (CX+Du^*|_{[\t,T]})dW(s), \\
         & dY(s)=-\big(A^\top Y+C^\top Z+QX+S^\top u^*|_{[\t,T]}\big)ds+ZdW(s), \\
         & X(\t)=X^*(\t), \q Y(T)=GX(T)
      \end{aligned}\right.$$
      satisfies
      $$B(s)^\top Y(s)+D(s)^\top Z(s)+S(s)X(s)+R(s)u^*|_{[\t,T]}(s)=0, \q\ae~s\in[\t,T],~\as$$
\end{enumerate}
To prove (a), let $v\in\cU[\t,T]$ be arbitrary and define the {\it zero-extension} of $v$ on $[t,T]$ as follows:
$$v_e(s)=\left\{\begin{aligned}
    & 0,    && s\in[t,\t), \\
    & v(s), && s\in[\t,T].
\end{aligned}\right.$$
Clearly, $v_e\in\cU[t,T]$. Denote by $X^\t$ and $X^t$ the solutions to the SDEs
$$\left\{\begin{aligned}
   dX^\t(s) &= (AX^\t+Bv)ds + (CX^\t+Dv)dW(s),\q s\in[\t,T], \\
   X^\t(\t) &= 0,
\end{aligned}\right.$$
and
$$\left\{\begin{aligned}
   dX^t(s) &= (AX^t+Bv_e)ds + (CX^t+Dv_e)dW(s),\q s\in[t,T], \\
    X^t(t) &= 0,
\end{aligned}\right.$$
respectively. Since the initial states of the above two SDEs are $0$ and $v_e=0$ on $[t,\t)$, we have
$$X^t(s)=0, \q s\in[t,\t]; \qq X^t(s)=X^\t(s), \q s\in[\t,T],$$
from which it follows that
\begin{align}\label{7-21:Jt=Jtao}
J(\t,0;v) &= \dbE\lt[\lan GX^\t(T),X^\t(T)\ran
             +\int_\t^T\llan\1n\begin{pmatrix}Q(s)&\1nS(s)^\top \\ S(s)&\1nR(s)\end{pmatrix}\1n
                               \begin{pmatrix}X^\t(s) \\ v(s)\end{pmatrix}\1n,
                               \begin{pmatrix}X^\t(s) \\ v(s)\end{pmatrix}\1n\rran ds\rt] \nn\\
          &= \dbE\lt[\lan GX^t(T),X^t(T)\ran
             +\int_t^T\llan\1n\begin{pmatrix}Q(s)&\1nS(s)^\top \\ S(s)&\1nR(s)\end{pmatrix}\1n
                              \begin{pmatrix}X^t(s) \\ v_e(s)\end{pmatrix}\1n,
                              \begin{pmatrix}X^t(s) \\ v_e(s)\end{pmatrix}\1n\rran ds\rt]= J(t,0;v_e).
\end{align}
Since by assumption, Problem (SLQ) is solvable at $(t,\xi)$, we obtain from \autoref{thm:SLQ-open-kehua} (i)
and relation \rf{7-21:Jt=Jtao} that
$$J(\t,0;v)= J(t,0;v_e) \ges0, \q\forall v\in\cU[\t,T].$$
To prove (b), let $(X^*,Y^*,Z^*)=\{(X^*(s),Y^*(s),Z^*(s));t\les s\les T\}$ be the adapted solution to
$$\left\{\begin{aligned}
   & dX^*(s)=(AX^*+Bu^*)ds + (CX^*+Du^*)dW(s), \\
   & dY^*(s)=-\big(A^\top Y^*+C^\top Z^*+QX^*+S^\top u^*\big)ds+Z^*dW(s), \\
   & X^*(t)=\xi, \q Y^*(T)=GX^*(T).
\end{aligned}\right.$$
Since $u^*\in\cU[t,T]$ is an optimal control at $(t,\xi)$,
we have by \autoref{thm:SLQ-open-kehua} (ii) that
$$B(s)^\top Y^*(s)+D(s)^\top Z^*(s)+S(s)X^*(s)+R(s)u^*(s)=0, \q\ae~s\in[t,T],~\as$$
Then assertion (b) follows from the fact that
$$(X(s),Y(s),Z(s)) = (X^*(s),Y^*(s),Z^*(s)), \q \t\les s\les T.$$
The proof is completed.
\end{proof}

\section{Properties of the Stochastic Value Flow $\h V(t,\xi)$}\label{sec:value-process}

We present in this section some properties of the stochastic value flow $\h V(t,\xi)$.
These include a quadratic representation of $\h V(t,\xi)$ in terms of a bounded,
$\dbS^n$-valued process $P=\{P(t);0\les t\les T\}$ as well as the left-continuity
of $t\mapsto P(t)$.
We shall see in \autoref{sec:SRE} that the sample paths of $P$ are actually continuous
and that $P$, together with another square-integrable process $\L=\{\L(t);0\les t\les T\}$,
satisfies a stochastic Riccati equation.

\ms

Let $e_1,\ldots,e_n$ be the standard basis for $\dbR^n$.
Recall that for a state-control pair $(X,u)=\{(X(s)$, $u(s));t\les s\les T\}$ corresponding to
the initial pair $(t,\xi)$, the associated adjoint BSDE is given by
\bel{adjoint}\left\{\begin{aligned} \sl
   dY(s) &= -\big(A^\top Y+C^\top Z+QX+S^\top u\big)ds+ZdW(s), \q s\in[t,T], \\
    Y(T) &= GX(T).
\end{aligned}\right.\ee
We have the following result.

\begin{proposition}\label{prop:ux=Ux} \sl
Let {\rm\ref{(A1)}--\ref{(A2)}} hold and let $t\in[0,T)$ be given.
Suppose that Problem {\rm(SLQ)} is solvable at the initial pair $(t,e_i)$ for every $1\les i\les n$.
Let $(X_i,u_i)=\{(X_i(s),u_i(s));t\les s\les T\}$ be an optimal pair with respect to $(t,e_i)$,
and let $(Y_i,Z_i)=\{(Y_i(s),Z_i(s));t\les s\les T\}$ be the adapted solution to the associated adjoint BSDE.
Then with
$$\BX=(X_1,\cdots,X_n), \q \BU=(u_1,\cdots,u_n), \q \BY=(Y_1,\cdots,Y_n), \q \BZ=(Z_1,\cdots,Z_n),$$
the $4$-tuple of matrix-valued processes $(\BX,\BU,\BY,\BZ)$ satisfies the FBSDE
$$\left\{\begin{aligned}
   & d\BX(s)=(A\BX+B\BU)ds + (C\BX+D\BU)dW(s), \\
   & d\BY(s)=-\big(A^\top\BY+C^\top\BZ+Q\BX+S^\top\BU\big)ds+\BZ dW(s), \\
   & \BX(t)=I_n, \q \BY(T)=G\BX(T),
\end{aligned}\right.$$
and is such that
\bel{6-12-stationarity} B^\top\BY +D^\top\BZ +S\BX +R\BU=0, \q\ae~\hb{on}~[t,T],~\as \ee
Moreover, the state-control pair $(\BX\xi,\BU\xi)=\{(\BX(s)\xi,\BU(s)\xi);t\les s\les T\}$ is optimal
with respect to $(t,\xi)$ for any $\xi\in L^\i_{\cF_t}(\Om;\dbR^n)$,
and $(\BY\xi,\BZ\xi)=\{(\BY(s)\xi,\BZ(s)\xi);t\les s\les T\}$ solves the adjoint BSDE \rf{adjoint} associated
with $(X,u)=(\BX\xi,\BU\xi)$.
\end{proposition}

\begin{proof}
The first assertion is an immediate consequence of \autoref{thm:SLQ-open-kehua}.
For the second assertion, we note that since $\xi$ is $\cF_t$-measurable and bounded, the pair
$$\big(X^*(s),u^*(s)\big)\deq\big(\BX(s)\xi,\BU(s)\xi\big); \q t\les s\les T$$
is square-integrable and satisfies the state equation
$$\left\{\begin{aligned}
   dX^*(s) &=(AX^*+Bu^*)ds + (CX^*+Du^*)dW(s), \q s\in[t,T],\\
    X^*(t) &=\xi.
\end{aligned}\right.$$
With the same reason, we see that the pair
$$ (Y^*(s),Z^*(s))\deq(\BY(s)\xi,\BZ(s)\xi); \q t\les s\les T $$
is the adapted solution to the adjoint BSDE associated with $(X^*,u^*)$:
$$\left\{\begin{aligned}
   dY^*(s) &=-\big(A^\top Y^*+C^\top Z^*+QX^*+S^\top u^*\big)ds +Z^*dW(s), \q s\in[t,T],\\
    Y^*(T) &=GX^*(T).
\end{aligned}\right.$$
Furthermore, \rf{6-12-stationarity} implies that
$$ B^\top Y^*+D^\top Z^*+SX^*+Ru^*= \big(B^\top\BY+D^\top\BZ+S\BX+R\BU\big)\xi=0,
   \q\ae~\hb{on}~[t,T],~\as $$
Thus by \autoref{thm:SLQ-open-kehua}, $(X^*,u^*)$ is optimal with respect to $(t,\xi)$.
\end{proof}

The following result shows that the stochastic value flow has a quadratic form.

\begin{theorem}\label{thm:V=Pxx} \sl
Let {\rm\ref{(A1)}--\ref{(A2)}} hold. If Problem {\rm(SLQ)} is solvable at $t$, then there exists
an $\dbS^n$-valued, $\cF_t$-measurable, integrable random variable $P(t)$ such that
\bel{hV}\h V(t,\xi) = \lan P(t)\xi,\xi\ran, \q\forall\xi\in L^\i_{\cF_t}(\Om;\dbR^n).\ee
\end{theorem}

\begin{proof}
Let $\{(X_i(s),u_i(s));t\les s\les T\}$ and $\{(\BX(s),\BU(s));t\les s\les T\}$ be as in \autoref{prop:ux=Ux}.
Then by \autoref{prop:ux=Ux}, the state-control pair $(\BX\xi,\BU\xi)$ is optimal with respect to $(t,\xi)$
for any $\xi\in L^\i_{\cF_t}(\Om;\dbR^n)$. Denoting
$$ \BM(T) =\BX(T)^\top G\BX(T),
\q \BN(s) =\begin{pmatrix}\BX(s) \\ \BU(s)\end{pmatrix}^{\1n\top}\2n
           \begin{pmatrix}Q(s)&S(s)^\top \\ S(s)&R(s)\end{pmatrix}
           \begin{pmatrix}\BX(s) \\ \BU(s)\end{pmatrix}, $$
we may write
\begin{align*}
   L(t,\xi;\BU\xi)
   &= \lan G\BX(T)\xi,\BX(T)\xi\ran
      +\int_t^T\llan\1n\begin{pmatrix}Q(s)&\1nS(s)^\top \\ S(s)&\1nR(s)\end{pmatrix}\1n
                       \begin{pmatrix}\BX(s)\xi \\ \BU(s)\xi\end{pmatrix}\1n,
                       \begin{pmatrix}\BX(s)\xi \\ \BU(s)\xi\end{pmatrix}\1n\rran ds \\
   &= \lan\BM(T)\xi,\xi\ran +\int_t^T\lan\BN(s)\xi,\xi\ran ds.
\end{align*}
Since $\xi$ is $\cF_t$-measurable, it follows that
$$\h V(t,\xi) = \dbE[L(t,\xi;\BU\xi)|\cF_t]
              = \lan\dbE\[\BM(T)+\int_t^T\BN(s)ds\bigm|\cF_t\]\xi,\xi\ran
         \equiv \lan P(t)\xi,\xi\ran.$$
The proof is completed.
\end{proof}

\begin{remark}\rm
So far we have established a number of results for Problem (SLQ) on
a deterministic interval $[t,T]$. We may also consider Problem (SLQ) on stochastic intervals $[\si,\t]$,
where $\si$ and $\t$ are $\dbF$-stopping times with $0\les\si\les\t\les T$.
With $t$ and $T$ respectively replaced by two finite stopping times $\si$ and $\t$, all the previous
results remain valid and can be proved using the same argument as before.
See \cite{Chen-Yong 2000, Chen-Yong 2001} for a similar consideration.
\end{remark}
From \autoref{crlry:existence-iff} (i), we see that $\cN_t\ges0$ (or equivalently, $[\![ \cN_tu,u ]\!]\ges0$
for all $u\in\cU[t,T]$) is a necessary condition for the existence of an optimal control, and from \autoref{crlry:existence-iff} (ii), we see that a sufficient condition guaranteeing the existence of a unique optimal control is
$$\cN_t\ges0 \q\hb{and}\q \cN_t~\hb{is invertible,}$$
which is equivalent to the uniform positive-definiteness of $\cN_t$, that is, there exists a constant $\d>0$ such that
\bel{uni-convex*} J(t,0;u) = [\![ \cN_tu,u]\!]\ges\d[\![u,u]\!] = \d\,\dbE\int_t^T|u(s)|^2ds, \q\forall u\in\cU[t,T].\ee

To carry out some further investigations of the stochastic value flow, let us suppose now that at the initial time $t=0$, the cost functional is uniformly convex; i.e., the following holds:
\begin{equation}\label{convexity:uni}
J(0,0;u)= [\![ \cN_0u,u ]\!] \ges \d\,\dbE\int_0^T|u(s)|^2ds, \q\forall u\in \cU[0,T],\q\hb{for some $\d>0$}.
\end{equation}
Such a condition implies that Problem (SLQ) is solvable at $t=0$ (see \autoref{crlry:existence-iff} (ii)).
The next result further shows that Problem (SLQ) is actually solvable at any stopping time $\t:\Om\to[0,T]$
when condition \rf{convexity:uni} holds.

\begin{proposition}\label{prop:t-tao} \sl Let {\rm\ref{(A1)}--\ref{(A2)}} hold. Suppose \rf{convexity:uni} holds. Then for any $\dbF$-stopping time $\t:\Om\to[0,T]$, we have
$$ J(\t,0;u) \ges \d\,\dbE\int_\t^T|u(s)|^2ds, \q\forall u\in\cU[\t,T].$$
Consequently, Problem {\rm(SLQ)} is uniquely solvable at $\t$.
\end{proposition}

\begin{proof}
Let $u\in \cU[\t,T]$ be arbitrary and define
$$u_e(s)=\left\{\begin{aligned}
    & 0,    && s\in[0,\t), \\
    & u(s), && s\in[\t,T].
\end{aligned}\right.$$
Processing exactly as in the proof of \autoref{crlry:Bellman-Principle} (the proof of (b)) with $t$,$v$ and $v_e$ replaced by $0$, $u$ and $u_e$, respectively, we obtain
$$J(\t,0;u) = J(0,0;u_e) \ges \d\,\dbE\int_0^T|u_e(s)|^2ds = \d\,\dbE\int_\t^T|u(s)|^2ds.$$
Thus, by \autoref{crlry:existence-iff} (ii), Problem {\rm(SLQ)} is uniquely solvable at $\t$.
\end{proof}

Under the conditions of \autoref{prop:t-tao}, Problem (SLQ) is solvable at any initial time $t\in[0,T]$.
Thus, according to \autoref{thm:V=Pxx}, there exists an $\dbF$-adapted process $P:[0,T]\times\Om\to\dbS^n$
such that
\begin{equation}\label{V-property-1*}
\h V(t,\xi) = \lan P(t)\xi,\xi\ran, \q\forall (t,\xi)\in[0,T]\times L^\i_{\cF_t}(\Om;\dbR^n).
\end{equation}
It is trivially seen that $P(T)=G$. Our next aim is to show that the process $P=\{P(t);0\les t\les T\}$
is bounded and left-continuous.
To this end, let $\t$ be an $\dbF$-stopping time with values in $(0,T]$ and denote by $\cS[0,\t)$
the set of $\dbF$-stopping times valued in $[0,\t)$. Let
$$\cD^\t=\big\{(\si,\xi)~|~\si\in\cS[0,\t),~\xi\in L^2_{\cF_\si}(\Om;\dbR^n)\big\},$$
and denote $\cU[\si,\t]=L_\dbF^2(\si,\t;\dbR^m)$ for $\si\in\cS[0,\t)$.
Consider the following {\it stopped} SLQ problem:

\ms

{\bf Problem (SLQ)$^\t$.} For any given initial pair $(\si,\xi)\in\cD^\t$, find a control $u^*\in\cU[\si,\t]$
such that the cost functional
$$ J^\t(\si,\xi;u) \deq \dbE\[\lan P(\t)X(\t),X(\t)\ran
      +\int_\si^\t\llan\1n\begin{pmatrix}Q(s)&\1nS(s)^\top \\ S(s)&\1nR(s)\end{pmatrix}\1n
                         \begin{pmatrix}X(s) \\ u(s)\end{pmatrix}\1n,
                         \begin{pmatrix}X(s) \\ u(s)\end{pmatrix}\1n\rran ds\] $$
is minimized subject to the state equation (over the stochastic interval $[\si,\t]$)
\bel{X-tau-kappa}\left\{\begin{aligned}
   dX(s) &=[A(s)X(s)+B(s)u(s)]ds+[C(s)X(s)+D(s)u(s)]dW(s), \q s\in[\si,\t],\\
  X(\si) &=\xi.
\end{aligned}\right.\ee

\begin{proposition}\label{prop:SLQtau} \sl Let {\rm\ref{(A1)}--\ref{(A2)}} hold. Suppose \rf{convexity:uni} holds. Then
\begin{enumerate}[\q\rm(i)]
\item for any $\si\in\cS[0,\t)$,
      $$ J^\t(\si,0;u) \ges \d\,\dbE\int_\si^\t|u(s)|^2ds, \q\forall u\in \cU[\si,\t]; $$
\item Problem {\rm(SLQ)$^\t$} is uniquely solvable at any $\si\in\cS[0,\t)$;
\item if $u^*\in\cU[\si,T]$ is an optimal control of Problem {\rm(SLQ)} at $(\si,\xi)\in\cD$, then the
      restriction $u^*|_{[\si,\t]}$ of $u^*$ to $[\si,\t]$ is an optimal control of Problem {\rm(SLQ)$^\t$} at the same initial pair $(\si,\xi)$;
\item the value flow $V^\t(\cd\,,\cd)$ of Problem {\rm(SLQ)$^\t$} admits the following form:
      $$V^\t(\si,\xi) = \dbE\lan P(\si)\xi,\xi\ran, \q\forall (\si,\xi)\in\cS[0,\t)\times L^\i_{\cF_\si}(\Om;\dbR^n).$$
\end{enumerate}
\end{proposition}

\begin{proof}
Fix an arbitrary stopping time $\si\in\cS[0,\t)$. For $\xi\in L^\i_{\cF_\si}(\Om;\dbR^n)$ and $u\in\cU[\si,\t]$,
let $X_1=\{X_1(s);\si\les s\les\t\}$ denote the corresponding solution to \rf{X-tau-kappa}.
Consider Problem (SLQ) for the initial pair $(\t,X_1(\t))$.
Since there exists a constant $\d>0$ such that \rf{convexity:uni} holds, Problem (SLQ) is solvable at $\t$
(\autoref{prop:t-tao}), and from \rf{V-property-1*} we see that
$$ \h V(\t,X_1(\t))=\lan P(\t)X_1(\t),X_1(\t)\ran. $$
Let $v^*\in\cU[\t,T]$ be an optimal control of Problem (SLQ) with respect to $(\t,X_1(\t))$ and
let $X_2^*=\{X_2^*(s);\t\les s\les T\}$ be the corresponding optimal state process. Define
$$[u\oplus v^*](s)=\left\{\begin{aligned}
    & u(s),   && s\in[\si,\t), \\
    & v^*(s), && s\in[\t,T].
\end{aligned}\right.$$
Obviously, the process $u\oplus v^*$ is in $\cU[\si,T]$, and the solution $X=\{X(s);\si\les s\les T\}$ to
$$\left\{\begin{aligned}
   dX(s) &=\{A(s)X(s)+B(s)[u\oplus v^*](s)\}ds \\
         &\hp{=\ } +\{C(s)X(s)+D(s)[u\oplus v^*](s)\}dW(s), \q s\in[\si,T],\\
    X(\si) &=\xi
\end{aligned}\right.$$
is such that
$$ X(s)=\left\{\begin{aligned}
    & X_1(s),   && s\in[\si,\t), \\
    & X_2^*(s), && s\in[\t,T].
\end{aligned}\right.$$
It follows that
\begin{align}\label{J-Jtau}
J(\si,\xi;u\oplus v^*)
  &= \dbE\[\lan GX(T),X(T)\ran
     +\int_\si^T\llan\1n\begin{pmatrix}Q(s)&\1nS(s)^\top \\ S(s)&\1nR(s)\end{pmatrix}\1n
                       \begin{pmatrix}X(s) \\ [u\oplus v^*](s)\end{pmatrix}\1n,
                       \begin{pmatrix}X(s) \\ [u\oplus v^*](s)\end{pmatrix}\1n\rran ds \] \nn\\
  &= \dbE\[\lan GX_2^*(T),X_2^*(T)\ran
     +\int_\t^T\llan\1n\begin{pmatrix}Q(s)&\1nS(s)^\top \\ S(s)&\1n R(s)\end{pmatrix}\1n
                       \begin{pmatrix}X_2^*(s) \\ v^*(s)\end{pmatrix}\1n,
                       \begin{pmatrix}X_2^*(s) \\ v^*(s)\end{pmatrix}\1n\rran ds\] \nn\\
  &\hp{=\ } +\dbE\[\int_\si^\t\llan\1n\begin{pmatrix}Q(s)&\1nS(s)^\top \\ S(s)&\1nR(s)\end{pmatrix}\1n
                                      \begin{pmatrix}X_1(s) \\ u(s)\end{pmatrix}\1n,
                                      \begin{pmatrix}X_1(s) \\ u(s)\end{pmatrix}\1n\rran ds \] \nn\\
  &= J(\t,X_1(\t);v^*) +\dbE\[\int_\si^\t\llan\1n\begin{pmatrix}Q(s)&\1nS(s)^\top \\ S(s)&\1nR(s)\end{pmatrix}\1n
                                                \begin{pmatrix}X_1(s) \\ u(s)\end{pmatrix}\1n,
                                                \begin{pmatrix}X_1(s) \\ u(s)\end{pmatrix}\1n\rran ds \] \nn\\
  &= \dbE\lan P(\t)X_1(\t),X_1(\t)\ran
     +\dbE\[\int_\si^\t\llan\1n\begin{pmatrix}Q(s)&\1nS(s)^\top \\ S(s)&\1nR(s)\end{pmatrix}\1n
                                \begin{pmatrix}X_1(s) \\ u(s)\end{pmatrix}\1n,
                                \begin{pmatrix}X_1(s) \\ u(s)\end{pmatrix}\1n\rran ds\] \nn\\
  &= J^\t(\si,\xi;u).
\end{align}
In particular, taking $\xi=0$ yields
$$J^\t(\si,0;u) =J(\si,0;u\oplus v^*) \ges \d\,\dbE\int_\si^T|[u\oplus v^*](s)|^2ds \ges \d\,\dbE\int_\si^\t|u(s)|^2ds. $$
This proves the first assertion.

\ms

The second assertion follows directly from (i) and  \autoref{crlry:existence-iff} (ii).

\ms

Finally, we look at (iii) and (iv). Observe first that relation \rf{J-Jtau} implies that
\bel{Jtau>} J^\t(\si,\xi;u) \ges \dbE\lan P(\si)\xi,\xi\ran. \ee
Suppose now that $u^*\in\cU[\si,T]$ is an optimal control of Problem (SLQ) at $(\si,\xi)$. Let $X^*=\{X^*(s);\si\les s\les T\}$ be the corresponding optimal state process,
that is, $X^*$ is the solution to
$$\left\{\begin{aligned}
   dX^*(s) &=\{A(s)X^*(s)+B(s)u^*(s)\}ds+\{C(s)X^*(s)+D(s)u^*(s)\}dW(s), \q s\in[\si,T],\\
    X^*(\si) &=\xi.
\end{aligned}\right.$$
Then by the principle of optimality (\autoref{crlry:Bellman-Principle}), the restriction $u^*|_{[\t,T]}$ of $u^*$ to $[\t,T]$ is optimal for Problem (SLQ) at $(\t,X^*(\t))$.
Replacing the processes $u$ and $v^*$ in \rf{J-Jtau} by $u^*|_{[\si,\t]}$ and $u^*|_{[\t,T]}$, respectively,
and noting that $u^*|_{[\si,\t]}\oplus u^*|_{[\t,T]}=u^*$, we obtain
\bel{Jtau=} J^\t(\si,\xi;u^*|_{[\si,\t]}) = J(\si,\xi;u^*) = \dbE\lan P(\si)\xi,\xi\ran. \ee
The last two assertions follow immediately from \rf{Jtau>} and \rf{Jtau=}.
\end{proof}

\begin{theorem}\label{thm:P-continuity} \sl
Under the hypotheses of \autoref{prop:t-tao}, the process $P=\{P(t);0\les t\les T\}$
in \rf{V-property-1*} is bounded and left-continuous.
\end{theorem}

\begin{proof}
We first prove that $P$ is bounded. By \autoref{prop:t-tao}, for any $t\in[0,T)$,
the operator $\cN_t$ defined by \rf{cN-def} satisfies
\begin{equation}\label{Nt>delta}
[\![ \cN_tu,u ]\!] = J(t,0;u) \ges \d\,[\![ u,u ]\!], \q\forall u\in \cU[t,T].
\end{equation}
This means $\cN_t$ is positive and invertible.
By \autoref{crlry:existence-iff} (ii), for any initial state $\xi\in L^\i_{\cF_t}(\Om;\dbR^n)$,
the corresponding optimal control is given by $u^*_{t,\xi} = -\cN_t^{-1}\cL_t\xi$.
Substituting $u^*_{t,\xi}$ into the representation \rf{J-rep} yields
\bel{P-rep}\dbE\lan P(t)\xi,\xi\ran =V(t,\xi) = \dbE\lan M(t)\xi,\xi\ran -[\![\cN_t^{-1}\cL_t\xi,\cL_t\xi ]\!],\ee
from which it follows immediately that
\bel{upper-bound}\dbE\lan P(t)\xi,\xi\ran \les \dbE\lan M(t)\xi,\xi\ran.\ee
On the other hand, combining \rf{Nt>delta} and \rf{P-rep}, together with
\rf{uni-bound-Lt}, we obtain
\bel{lower-bound*}\ba{ll}
\ns\ds\dbE\lan P(t)\xi,\xi\ran \ges \dbE\lan M(t)\xi,\xi\ran -\d^{-1}[\![ \cL_t\xi,\cL_t\xi ]\!]\\
\ns\ds\qq\qq\q\ges \dbE\lan M(t)\xi,\xi\ran -\d^{-1}K\,\dbE|\xi|^2
                                             = \dbE\lan [M(t)-\d^{-1}K I_n]\xi,\xi\ran.\ea\ee
Since $\xi\in L^\i_{\cF_t}(\Om;\dbR^n)$ is arbitrary, we conclude that
$$M(t)-\d^{-1}K I_n \les P(t) \les M(t).$$
The boundedness of $P$ follows by noting that $M$ is bounded (\autoref{prop:M-bounded}).

\ms

We next show that $P$ is left-continuous. Without loss of generality, we consider only the left-continuity at $t=T$.
The case of $t\in(0,T)$ can be treated in a similar manner by considering Problem (SLQ)$^t$.
We notice first that, thanks to \rf{upper-bound} and \rf{lower-bound*}, for any initial pair
$(t,\xi)\in[0,T)\times L^\i_{\cF_t}(\Om;\dbR^n)$,
\begin{equation}\label{7-24}
\dbE\lan M(t)\xi,\xi\ran -\d^{-1}[\![ \cL_t\xi,\cL_t\xi ]\!] \les \dbE\lan P(t)\xi,\xi\ran \les \dbE\lan M(t)\xi,\xi\ran.
\end{equation}
Using \rf{Lt-rep*} and denoting $\dbL(s)=B(s)^\top\dbY(s)+D(s)^\top\dbZ(s)+S(s)\dbX(s)$,
we can rewrite $[\![ \cL_t\xi,\cL_t\xi ]\!]$ as
$$[\![ \cL_t\xi,\cL_t\xi ]\!] = \dbE\int_t^T\llan[\dbL(s)\dbX^{-1}(t)]^\top[\dbL(s)\dbX^{-1}(t)]\xi,\xi\rran ds. $$
Since $M(t)$, $P(t)$, and $X(t)$ are $\cF_t$-measurable and $\xi\in L^\i_{\cF_t}(\Om;\dbR^n)$ is arbitrary,
we can take conditional expectations with respect to $\cF_t$ in \rf{7-24} to obtain
$$M(t) -\d^{-1}[\dbX^{-1}(t)]^\top\dbE\[\int_t^T\dbL(s)^\top\dbL(s)ds\bigm|\cF_t\]\dbX^{-1}(t) \les P(t) \les M(t).$$
Letting $t\uparrow T$ and using the conditional dominated convergence theorem, we obtain
$$ \lim_{t\uparrow T}P(t) =\lim_{t\uparrow T}M(t) =G =P(T). $$
The proof is completed.
\end{proof}

\begin{corollary} \sl
Let {\rm\ref{(A1)}--\ref{(A2)}} hold. Suppose \rf{convexity:uni} holds. Then the stochastic value flow $\h V(\cd\,,\cd)$ of Problem {\rm(SLQ)} admits the following form over $\cD$:
$$ \h V(t,\xi)=\lan P(t)\xi,\xi\ran, \q\forall(t,\xi)\in\cD. $$
\end{corollary}

\begin{proof}
By Theorem \ref{thm:P-continuity}, the process $P$ is bounded. Hence, we may extend the representation
\rf{V-property-1*} from $L^\i_{\cF_t}(\Om;\dbR^n)$ to $\cX_t\equiv L^2_{\cF_t}(\Om;\dbR^n)$.
\end{proof}

\section{Riccati Equation and Closed-Loop Representation}\label{sec:SRE}

In this section we establish the solvability of the stochastic Riccati equation (SRE, for short)
\bel{Ric}\left\{\begin{aligned}
   dP(t) &=-\big[PA+A^\top P +C^\top PC +\L C+C^\top\L +Q \\
         &\hp{=-\[} -(PB +C^\top PD +\L D +S^\top)(R+D^\top PD)^{-1}(B^\top P +D^\top PC +D^\top\L +S)\big]dt\\
         &\qq~+\L dW(t),\qq\qq\qq\qq t\in[0,T],\\
    P(T) &=G,
\end{aligned}\right.\ee
and derive the closed-loop representation of (open-loop) optimal controls. We have seen from previous sections that the convexity
\begin{equation}\label{7-24:convex}
J(0,0;u)= [\![ \cN_0u,u ]\!] \ges 0, \q\forall u\in\cU[0,T]
\end{equation}
is necessary for the solvability of Problem (SLQ) (\autoref{crlry:existence-iff} (i)),
and that the uniform convexity \rf{convexity:uni}, a slightly stronger condition than \rf{7-24:convex},
is sufficient for the existence of an optimal control for any initial pair (\autoref{prop:t-tao}).
In this section we shall prove that the SRE \rf{Ric} is uniquely solvable under \rf{convexity:uni}
and that the first component of its solution is exactly the process $P$ appeared in \rf{V-property-1*}.
As a by-product, the (open-loop) optimal control is represented as a linear feedback of the state.

\medskip

The main result of this section can be stated as follows.

\begin{theorem}\label{thm:main} \sl
Let {\rm\ref{(A1)}--\ref{(A2)}} hold. Suppose \rf{convexity:uni} holds.
Then Problem {\rm(SLQ)} is uniquely solvable and the SRE \rf{Ric} admits a unique adapted solution $(P,\L)$, and for some $\l>0$, the following holds:
\bel{7-28:uniconv}R+D^\top PD\ges\l I_m,\qq\ae~\hb{on}~[0,T],~\as \ee
Moreover, the unique optimal control $u^*_{t,\xi}=\{u^*_{t,\xi}(s);t\les s\les T\}$ corresponding to any $(t,\xi)\in\cS[0,T)\times L^\i_{\cF_t}(\Om;\dbR^n)$ takes the following linear state feedback form:
$$u^*_{t,\xi}(s) = \Th(s)X^*(s); \q s\in[t,T],$$
where $\Th$ is defined by
\bel{7-29:Th} \Th=-(R+D^\top PD)^{-1}(B^\top P +D^\top PC +D^\top\L +S), \ee
and $X^*=\{X^*(s);t\les s\les T\}$ is the solution of the
closed-loop system
$$\left\{\begin{aligned}
  dX^*(s) &=[A(s)+B(s)\Th(s)]X^*(s)ds +[C(s)+D(s)\Th(s)]X^*(s) dW(s), \q s\in[t,T],\\
   X^*(t) &=\xi.
\end{aligned}\right.$$
\end{theorem}

Because some preparations are needed, we defer the proof of \autoref{thm:main} to the end of this section.
The preparation for the proof starts with the following result,
which plays a crucial role in the sequel.

\begin{theorem}\label{thm:BX-invertible} \sl
Let {\rm\ref{(A1)}--\ref{(A2)}} hold, and let $e_1,\cdots,e_n$ be the standard basis for $\dbR^n$.
Suppose \rf{convexity:uni} holds.
Let $X_i=\{X_i(s);0\les s\les T\}$ be the (unique) optimal state process with respect to the initial
pair $(t,\xi)=(0,e_i)$.
Then the $\dbR^{n\times n}$-valued process $\BX=\{\BX(s)\deq(X_1(s),\ldots,X_n(s));0\les s\les T\}$ is invertible.
\end{theorem}

\begin{proof}
Let $u_i\in\cU[0,T]$ be the unique optimal control with respect to $(0,e_i)$ so that
$$\left\{\begin{aligned}
   dX_i(s) &=(AX_i+Bu_i)ds + (CX_i+Du_i)dW(s), \q s\in[0,T], \\
    X_i(0) &=e_i.
\end{aligned}\right.$$
Then with $\BU(s)=(u_1(s),\ldots,u_n(s))$, we have
\bel{SDE-X}\left\{\begin{aligned}
   d\BX(s) &=(A\BX+B\BU)ds + (C\BX+D\BU)dW(s), \q s\in[0,T], \\
    \BX(0) &=I_n.
\end{aligned}\right.\ee
Define the stopping time (at which $\BX$ is not invertible for the first time)
$$\th(\om)=\inf\{s\in[0,T];~\det(\BX(s,\om))=0\},$$
where we employ the convention that the infimum of the empty set is infinity.
In order to prove that $\BX$ is invertible, it suffices to show that $\dbP(\th=\i)=1$, or equivalently, that the set
$$ \G = \{\om\in\Om;~\th(\om)\les T\}$$
has probability zero. Suppose the contrary and set $\t=\th\wedge T$. Then $\t$ is also a stopping time and $0<\t\les T$.
Since $\t=\th$ on $\G$, by the definition of $\th$, $\BX(\t)$ is not invertible on $\G$. Thus, we can choose an
$\dbS^n$-valued, $\cF_\t$-measurable, positive semi-definite random matrix $H$ with $|H|=1$ on $\G$ such that
$$ H(\om)\BX(\t(\om),\om)=0, \q\forall \om\in\Om. $$
Let $P$ be the bounded, left-continuous process in \rf{V-property-1*}.
We introduce the following auxiliary cost functional:
$$\cl{J^\t}(\si,\xi;u)\deq J^\t(\si,\xi;u)+\dbE\lan HX(\t),X(\t)\ran.$$
Consider the problem of minimizing the above auxiliary cost functional subject to the state equation \rf{X-tau-kappa},
which will be called Problem $\cl{\hb{(SLQ)$^\t$}}$ and whose value flow will be denoted by $\cl{V^\t}(\cd\,,\cd)$.
We have the following facts:
\begin{enumerate}[(1)]
\item {\sl For any $\si\in\cS[0,\t)$,
      \begin{equation*}
         \cl{J^\t}(\si,0;u) \ges J^\t(\si,0;u) \ges \d\,\dbE\int_\si^\t|u(s)|^2ds,
         \q\forall u\in\cU[\si,\t].
      \end{equation*}
      Consequently, both Problems {\rm(SLQ)$^\t$} and $\cl{\hb{\rm(SLQ)$^\t$}}$ are uniquely solvable at any $\si\in\cS[0,\t)$.}

      Indeed, the first inequality is true since $H$ is positive semi-definite,
      and the second inequality is immediate from \autoref{prop:SLQtau} (i).

\item {\sl The restriction $u_i^\t=u_i|_{[0,\t]}$ of $u_i$ to $[0,\t]$ is optimal for both Problems {\rm(SLQ)$^\t$}
      and $\cl{\hb{\rm(SLQ)$^\t$}}$ at the same initial pair $(0,e_i)$.}

      Indeed, the fact that $u_i^\t$ is   optimal for Problem (SLQ)$^\t$ at $(0,e_i)$ is a direct
      consequence of \autoref{prop:SLQtau} (iii).
      According to \autoref{thm:SLQ-open-kehua}, to prove that $u_i^\t$ is also optimal for Problem $\cl{\hb{\rm(SLQ)$^\t$}}$
      at $(0,e_i)$, it suffices to show that the adapted solution $(X_i^\t,Y_i^\t,Z_i^\t)$ to the FBSDE
      \bel{FBSDE-XYZ_itau}\left\{\begin{aligned}
          & dX_i^\t(s)=(AX_i^\t+Bu_i^\t)ds + (CX_i^\t+Du_i^\t)dW(s), \q s\in[0,\t], \\
          & dY_i^\t(s)=-\big(A^\top Y_i^\t + C^\top Z_i^\t + QX_i^\t + S^\top u_i^\t\big)ds + Z_i^\t dW(s), \q s\in[0,\t],\\
          & X_i^\t(0)=e_i, \q Y_i^\t(\t)=[P(\t)+H]X_i^\t(\t)
      \end{aligned}\right.\ee
      satisfies
      \bel{XYZ_itau=0} B^\top Y_i^\t +D^\top Z_i^\t +SX_i^\t + Ru_i^\t =0. \ee
      We observe first that $X_i^\t(s) = \BX(s)e_i$ for $0\les s\les\t$.
      Thus, by the choice of $H$, we have
      \bel{7-26:HXitao=0} HX_i^\t(\t)=H\BX(\t)e_i=0. \ee
      It follows that \rf{FBSDE-XYZ_itau} is equivalent to
      $$\left\{\begin{aligned}
          & dX_i^\t(s)=(AX_i^\t+Bu_i^\t)ds + (CX_i^\t+Du_i^\t)dW(s), \q s\in[0,\t], \\
          & dY_i^\t(s)=-\big(A^\top Y_i^\t + C^\top Z_i^\t + QX_i^\t + S^\top u_i^\t\big)ds + Z_i^\t dW(s), \q s\in[0,\t],\\
          & X_i^\t(0)=e_i, \q Y_i^\t(\t)=P(\t)X_i^\t(\t),
      \end{aligned}\right.$$
      which is exactly the FBSDE associated with Problem (SLQ)$^\t$.
      Since $u_i^\t$ is an optimal control of Problem (SLQ)$^\t$ at $(0,e_i)$,
      we obtain \rf{XYZ_itau=0} by using \autoref{thm:SLQ-open-kehua} again.
\end{enumerate}
By fact (1), for Problem $\cl{\hb{(SLQ)$^\t$}}$ there exists a bounded, left-continuous process
$\bar P=\{\bar P(s);0\les s\les\t\}$ such that
$$\cl{V^\t}(\si,\xi)=\lan\bar P(\si)\xi,\xi\ran, \q\forall (\si,\xi)\in\cS[0,\t)\times L^\i_{\cF_\si}(\Om;\dbR^n).$$
By fact (2), we see that $(X_i^\t,u_i^\t)=\{(X_i^\t(s),u_i^\t(s));0\les s\les \t\}$ is the optimal state-control
pair for both Problem (SLQ)$^\t$ and Problem $\cl{\hb{(SLQ)$^\t$}}$ at $(0,e_i)$. Set
\begin{align*}
   \BX^\t &=\{\BX^\t(s)\deq(X_1^\t(s),\ldots,X_n^\t(s));0\les s\les \t\}, \\
   \BU^\t &=\{\BU^\t(s)\deq(u_1^\t(s),\ldots,u_n^\t(s));0\les s\les \t\},
\end{align*}
and take an arbitrary $x\in\dbR^n$. Then by \autoref{prop:ux=Ux}, $(\BX^\t x,\BU^\t x)$ is the optimal
state-control pair for both Problem (SLQ)$^\t$ and Problem $\cl{\hb{(SLQ)$^\t$}}$ at $(0,x)$.
Furthermore, by the principle of optimality (\autoref{crlry:Bellman-Principle}), the pair
$$(\BX^\t(s)x,\BU^\t(s)x); \q t\les s\les \t$$
remains optimal at $(t,\BX^\t(t)x)$ for any $0\les t<\t$.
Thus, noting that $H\BX^\t(\t)=0$ by \rf{7-26:HXitao=0}, we have
\begin{align*}
\cl{V^\t}(t,\BX^\t(t)x)
   &= \cl{J^\t}(t,\BX^\t(t)x;\BU^\t x) =J^\t(t,\BX^\t(t)x;\BU^\t x) + \dbE\lan H\BX^\t(\t)x,\BX^\t(\t)x\ran \\
   &= J^\t(t,\BX^\t(t)x;\BU^\t x) =\h V(t,\BX^\t(t)x).
\end{align*}
Noting that $\BX^\t(t)=\BX(t)$ for $0\les t\les\t$, we obtain from the above that
$$\lan\bar P(t)\BX(t)x,\BX(t)x\ran =\cl{V^\t}(t,\BX^\t(t)x) =\h V(t,\BX^\t(t)x) =\lan P(t)\BX(t)x,\BX(t)x\ran.$$
Since $x\in\dbR^n$ is arbitrary, it follows that
$$\BX(t)^\top P(t)\BX(t) = \BX(t)^\top\bar P(t)\BX(t); \q 0\les t <\t.$$
By the definition of $\t$, $\BX$ is invertible on $[0,\t)$. Hence,
\bel{P=tiP} P(t) = \bar P(t); \q 0\les t <\t. \ee
On the other hand, $\bar P(\t)= P(\t)+H$, and both $P$ and $\bar P$ are left-continuous.
Letting $t\uparrow\t$ in \rf{P=tiP} then yields a contradiction: $P(\t)= P(\t)+H$, since $|H|=1$ on $\G$.
\end{proof}

The next result establishes the unique solvability of the SRE \rf{Ric}.

\begin{theorem}\label{thm:Riccati} \sl
Let {\rm\ref{(A1)}--\ref{(A2)}} hold. Suppose \rf{convexity:uni} holds.
Then the stochastic Riccati equation \rf{Ric} admits a unique adapted solution
$(P,\L)\in L_\dbF^\i(\Om;C([0,T];\dbS^n))\times L_\dbF^2(0,T;\dbS^n)$ such that
\rf{7-28:uniconv} holds for some constant $\l>0$.
\end{theorem}

The proof of \autoref{thm:Riccati} proceeds through several lemmas.
As a preparation, we note first that by \autoref{prop:t-tao}, Problem (SLQ) is uniquely solvable
under the assumptions of \autoref{thm:Riccati}.
Let $(X_i,u_i)=\{(X_i(s),u_i(s));0\les s\les T\}$ be the unique optimal pair with respect to $(0,e_i)$,
and let $(Y_i,Z_i)=\{(Y_i(s),Z_i(s));0\les s\les T\}$ be the adapted solution to the adjoint BSDE
associated with $(X_i,u_i)$.
According to \autoref{prop:ux=Ux}, the $4$-tuple $(\BX,\BU,\BY,\BZ)$ defined by
\begin{equation}\label{7-29:notation}
  \BX=(X_1,\cdots,X_n), \q \BU=(u_1,\cdots,u_n), \q \BY=(Y_1,\cdots,Y_n), \q \BZ=(Z_1,\cdots,Z_n),
\end{equation}
satisfies the FBSDE
$$\left\{\begin{aligned}
   & d\BX(s)=(A\BX+B\BU)ds + (C\BX+D\BU)dW(s), \\
   & d\BY(s)=-\big(A^\top\BY+C^\top\BZ+Q\BX+S^\top\BU\big)ds+\BZ dW(s), \\
   & \BX(0)=I_n, \q \BY(T)=G\BX(T),
\end{aligned}\right.$$
and is such that
\bel{6-15-stationarity} B^\top\BY +D^\top\BZ +S\BX +R\BU=0, \q\ae~\hb{on}~[0,T],~\as \ee
Furthermore, \autoref{thm:BX-invertible} shows that the process $\BX=\{\BX(s);0\les s\les T\}$ is invertible,
and \textcolor{blue}{Theorems} \ref{thm:V=Pxx} and \ref{thm:P-continuity} imply that there exists a bounded,
left-continuous, $\dbF$-adapted process $P:[0,T]\times\Om\to\dbS^n$ such that \rf{V-property-1*} holds.

\begin{lemma} \sl
Under the assumptions of \autoref{thm:Riccati}, we have
\bel{Y=PX} P(t)=\BY(t)\BX(t)^{-1}, \q\forall t\in[0,T]. \ee
\end{lemma}

\begin{proof}
Let $x\in\dbR^n$ be arbitrary and set
\begin{align*}
   (X^*,u^*) &= \{(\BX(s)x,\BU(s)x);0\les s\les T\}, \\
   (Y^*,Z^*) &= \{(\BY(s)x,\BZ(s)x);0\les s\les T\}.
\end{align*}
From \autoref{prop:ux=Ux} we see that $(X^*,u^*)$ is an optimal pair with respect to $(0,x)$,
and that $(Y^*,Z^*)$ is the adapted solution to the adjoint BSDE associated with $(X^*,u^*)$.
For any $t\in[0,T]$, the principle of optimality (\autoref{crlry:Bellman-Principle}) shows that
the restriction $(X^*|_{[t,T]},u^*|_{[t,T]})$ of $(X^*,u^*)$ to $[t,T]$ remains optimal at $(t,X^*(t))$. Thus, we have by \autoref{crlry:V=Y(t)X(t)} that
$$\h V(t,X^*(t)) = \lan Y^*(t),X^*(t)\ran. $$
Because of \rf{V-property-1*}, the above yields
\begin{align*}
x^\top\BX(t)^\top P(t)\BX(t)x
  & =\lan P(t)\BX(t)x,\BX(t)x\ran =\lan P(t)X^*(t),X^*(t)\ran  =\lan Y^*(t),X^*(t)\ran \\
  & =\lan\BY(t)x,\BX(t)x\ran =x^\top\BX(t)^\top\BY(t)x.
\end{align*}
Since $x\in\dbR^n$ is arbitrary, we conclude that $\BX(t)^\top P(t)\BX(t)=\BX(t)^\top\BY(t)$.
The desired result then follows from the fact that $\BX$ is invertible.
\end{proof}

\begin{lemma}\label{lmm:pre-Ric} \sl
With the assumptions of \autoref{thm:Riccati} and the notation
\bel{def-Pi-L}\begin{aligned}
  \Th(t) &=\BU(t)\BX(t)^{-1}, \q \Pi(t)=\BZ(t)\BX(t)^{-1},\\
   \L(t) &=\Pi(t)-P(t)[C(t)+D(t)\Th(t)]; \q 0\les t\les T,
\end{aligned}\ee
the pair $(P,\L)$ satisfies the following BSDE:
\bel{6-15:dP=}\left\{\begin{aligned}
  dP(t) &=\big[-PA-A^\top P-C^\top PC-\L C-C^\top\L -Q \\
        &\hp{=\big[} -(PB+C^\top PD+\L D+S^\top)\Th\big]dt +\L dW(t), \q t\in[0,T], \\
   P(T) &=G.
\end{aligned}\right.\ee
Moreover, $\L=\L^\top$ and the following relation holds:
\bel{6-16:relation} B^\top P +D^\top PC +D^\top\L +S +(R+D^\top PD)\Th =0, \q\ae~\hb{on}~[0,T],~\as \ee
\end{lemma}

\begin{proof}
First of all, from \rf{V-property-1*} we see that
$$\lan G\xi,\xi\ran=\h V(T,\xi)=\lan P(T)\xi,\xi\ran, \q\forall\xi\in L^\i_{\cF_T}(\Om;\dbR^n),$$
which leads to $P(T)=G$. Since $\BX=\{\BX(s);0\les s\les T\}$ satisfies the SDE \rf{SDE-X} and is invertible,
It\^o's formula implies that its inverse $\BX^{-1}$ also satisfies a certain SDE. Suppose that
$$d\BX(s)^{-1}=\Xi(s)ds+\D(s)dW(s), \q s\in[0,T],$$
for some progressively processes $\{\Xi(s);0\les s\les T\}$ and $\{\D(s);0\les s\les T\}$.
Then by It\^o's formula and using \rf{SDE-X} and \rf{def-Pi-L}, we have
\begin{align*}
0 &= d\big(\BX\BX^{-1}\big) = \big[(A\BX+B\BU)\BX^{-1} + \BX\Xi + (C\BX+D\BU)\D\big]ds \\
  &\hp{= d\big(\BX\BX^{-1}\big) =~} +\big[(C\BX+D\BU)\BX^{-1} + \BX\D\big]dW(s) \\
  &=[A+B\Th + \BX\Xi + (C\BX+D\BU)\D]ds + (C+D\Th+\BX\D)dW(s).
\end{align*}
Thus, it is necessary that $\D=-\BX^{-1}(C+D\Th)$ and
\begin{align*}
  \Xi &= -\BX^{-1}[A+B\Th + (C\BX+D\BU)\D] \\
      &= -\BX^{-1}[A+B\Th-C(C+D\Th)-D\Th(C+D\Th)]  \\
      &= \BX^{-1}\big[(C+D\Th)^2-A-B\Th\big].
\end{align*}
Applying It\^{o}'s formula to the right-hand side of \rf{Y=PX} and then substituting for $\Xi$ and $\D$, we have
\begin{eqnarray*}
dP \3n&=&\3n -\,(A^\top\BY+C^\top\BZ+Q\BX +S^\top\BU)\BX^{-1}dt+\BZ\BX^{-1}dW+\,\BY\Xi dt + \BY\D dW + \BZ\D dt \nn\\
   \3n&=&\3n \big\{-A^\top P-C^\top\Pi-Q-S^\top\Th+P[(C+D\Th)^2-A-B\Th]-\,\Pi(C+D\Th)\big\}dt \\
   \3n&~&\3n  +[\Pi-P(C+D\Th)]dW \\
   \3n&=&\3n \big[-PA-A^\top P-C^\top PC-\L C-C^\top\L -Q-\,(PB+C^\top PD+\L D+S^\top)\Th\big]dt+\L dW.
\end{eqnarray*}
Recall that the process $P$ is symmetric; i.e., $P=P^\top$.
By comparing the diffusion coefficients of the SDEs for $P$ and $P^\top$, we conclude that
$$\L(t)=\L(t)^\top; \q 0\les t\les T.$$
Furthermore, \rf{def-Pi-L} and \rf{6-15-stationarity} imply that
\begin{align*}
  & B^\top P +D^\top PC +D^\top\L +S +(R+D^\top PD)\Th = B^\top P +D^\top\Pi +S +R\Th \\
  &\q= \big(B^\top\BY +D^\top\BZ +S\BX +R\BU\big)\BX^{-1} =0.
\end{align*}
The proof is completed.
\end{proof}

\begin{lemma}\label{lmm:R+DPD>I} \sl
Under the assumptions of \autoref{thm:Riccati}, we have
\bel{6-15:R+DPD>} R+D^\top PD \ges \d I_m, \q\ae~\hb{on}~[0,T],~\as \ee
\end{lemma}

\begin{proof}
The proof will be accomplished in several steps.

\ms

{\it Step 1:}
Let us temporarily assume that the processes $\Th=\{\Th(s);0\les s\les T\}$ and
$\L=\{\L(s);0\les s\les T\}$ defined by \rf{def-Pi-L} satisfy
\bel{Th-L-bounded}\esssup_{\om\in\Om}\int_0^T\[|\Th(s,\om)|^2+|\L(s,\om)|^2\]ds <\i.\ee
Take an arbitrary control $v\in\cU[0,T]$ and consider the SDE
\bel{7-27:X}\left\{\begin{aligned}
   dX(s) &=[(A+B\Th)X+Bv]ds +[(C+D\Th)X+Dv]dW(s), \q s\in[0,T],\\
    X(0) &=0.
\end{aligned}\right.\ee
By \autoref{lmm:sol-F&B-SDE}, the solution $X$ of \rf{7-27:X} belongs to the space
$L_\dbF^2(\Om;C([0,T];\dbR^n))$ and hence
\bel{u=UX+v} u\deq \Th X+v \in\cU[0,T]. \ee
Note that with the control defined by \rf{u=UX+v}, the solution to the state equation
$$\left\{\begin{aligned}
   dX(s) &=(AX+Bu)ds +(CX+Du)dW(s), \q s\in[0,T],\\
    X(0) &=0
\end{aligned}\right.$$
coincides with the solution $X$ to \rf{7-27:X}. Using \rf{6-15:dP=}, we obtain by It\^{o}'s rule that
\begin{align*}
d\lan PX,X\ran
   &= \big[-\lan QX,X\ran-\lan(PB+C^\top PD+\L D+S^\top)\Th X,X\ran \\
   &\hp{=\big[\1n} +2\lan(PB+C^\top PD+\L D)u,X\ran + \lan D^\top PDu,u\ran\big]ds \\
   &\hp{=\big[\1n} +\big[\lan\L X,X\ran+2\lan P(CX+Du),X\ran\big]dW,
\end{align*}
from which it follows that
\begin{align*}
\dbE\lan GX(T),X(T)\ran
   &= \dbE\int_0^T\[-\lan QX,X\ran-\lan(PB+C^\top PD+\L D+S^\top)\Th X,X\ran \\
   &\hp{=\dbE\int_0^T\[\ } +2\lan(PB+C^\top PD+\L D)u,X\ran + \lan D^\top PDu,u\ran\]ds.
\end{align*}
Substituting this into the cost functional yields
\begin{align*}
J(0,0;u)
   &= \dbE\int_0^T\[-\lan(PB+C^\top PD+\L D+S^\top)\Th X,X\ran \\
   &\hp{=\dbE\int_0^T\[\ } +2\lan(PB+C^\top PD+\L D+S^\top)u,X\ran + \lan (R+D^\top PD)u,u\ran\]ds.
\end{align*}
Using \rf{6-16:relation} and \rf{u=UX+v}, we can further obtain
$$ J(0,0;u) = \dbE\int_0^T\lan (R+D^\top PD)(u-\Th X),u-\Th X\ran ds
            = \dbE\int_0^T\lan (R+D^\top PD)v,v\ran ds. $$
Because by assumption, $J(0,0;u)\ges0$ for all $u\in\cU[0,T]$, we conclude from the last equation that
\bel{6-16:R+DPD>0} R+D^\top PD \ges 0, \q\ae~\hb{on}~[0,T],~\as \ee

\ss

{\it Step 2:}
We now prove that \rf{6-16:R+DPD>0} is still valid without the additional assumption \rf{Th-L-bounded}.
Here, the key idea is to employ a localization technique so that the preceding argument can be applied
to a certain stopped SLQ problem. More precisely, we define for each $k\ges1$ the stopping time
(with the convention $\inf\varnothing=\i$)
$$ \t_k = \inf\lt\{t\in[0,T]\bigm|\int_0^t\big(|\Th(s)|^2+|\L(s)|^2\big)ds\ges k\rt\}$$
and consider the corresponding Problem (SLQ)$^{\t_k}$.
Take an arbitrary control $v\in\cU[0,T]$ and consider the following SDE over $[0,\t_k]$:
\bel{X:tauk}\left\{\begin{aligned}
   dX(s) &=[(A+B\Th)X+Bv]ds +[(C+D\Th)X+Dv]dW(s), \q s\in[0,\t_k],\\
    X(0) &=0.
\end{aligned}\right.\ee
Since by the definition of $\t_k$,
$$\int_0^{\t_k}\[|\Th(s)|^2+|\L(s)|^2\]ds\les k,$$
we see from \autoref{lmm:sol-F&B-SDE} that the solution $X$ of \rf{X:tauk} belongs to the space
$L_\dbF^2(\Om;C([0,\t_k];\dbR^n))$ and hence
$$ u\deq \Th X+v \in\cU[0,\t_k]. $$
Then we may proceed as in Step $1$ to obtain
$$ J^{\t_k}(0,0;u) = \dbE\int_0^{\t_k}\lan (R+D^\top PD)v,v\ran ds. $$
Since by \autoref{prop:SLQtau} (i) $J^{\t_k}(0,0;u)\ges0$ for all $u\in\cU[0,\t_k]$ and $v\in\cU[0,T]$
is arbitrary, we conclude that
\bel{7-28:R+DPD} R+D^\top PD \ges 0, \q\ae~\hb{on}~[0,\t_k],~\as \ee
Because the processes $\BU=\{\BU(s);0\les s\les T\}$ and $\BZ=\{\BZ(s);0\les s\les T\}$ are square-integrable,
$\BX^{-1}=\{\BX(s)^{-1};0\les s\les T\}$ is continuous, and $P,C,D$ are bounded, we see from \rf{def-Pi-L} that
$$\int_0^T\[|\Th(s)|^2+|\L(s)|^2\]ds<\i, \q\as$$
This implies that $\lim_{k\to\i}\t_k=T$ almost surely.
Letting $k\to\i$ in \rf{7-28:R+DPD} then results in \rf{6-16:R+DPD>0}.

\ms

{\it Step 3:}
In order to obtain the stronger property \rf{6-15:R+DPD>}, we take an arbitrary but fixed $\e\in(0,\d)$
and consider the SLQ problem of minimizing
\begin{align*}
J_\e(t,\xi;u) &= J(t,\xi;u) -\e\dbE\int_t^T|u(s)|^2ds \\
              &= \dbE\lt[\lan GX(T),X(T)\ran
                 +\int_t^T\llan\1n\begin{pmatrix}Q(s)&\1nS(s)^\top \\ S(s)&\1nR(s)-\e I_m\end{pmatrix}\1n
                                  \begin{pmatrix}X(s) \\ u(s)\end{pmatrix}\1n,
                                  \begin{pmatrix}X(s) \\ u(s)\end{pmatrix}\1n\rran ds\rt]
\end{align*}
subject to the state equation \rf{state}. Clearly, with $\d$ replaced by $\d-\e$,
the assumptions of \autoref{thm:Riccati} still hold for the new cost functional $J_\e(t,\xi;u)$.
Thus, with $P_\e$ denoting the process such that
$$ V_\e(t,\xi) \deq \inf_{u\in\cU[t,T]} J_\e(t,\xi;u)= \lan P_\e(t)\xi,\xi\ran,
   \q\forall (t,\xi)\in[0,T]\times L^\i_{\cF_t}(\Om;\dbR^n), $$
we have by the previous argument that
$$ (R-\e I_m) +D^\top P_\e D \ges 0, \q\ae~\hb{on}~[0,T],~\as $$
Since by the definition of $J_\e(t,\xi;u)$,
$$ V(t,\xi) =\inf_{u\in\cU[t,T]} J(t,\xi;u)\ges \inf_{u\in\cU[t,T]} J_\e(t,\xi;u) =V_\e(t,\xi),
   \q\forall (t,\xi)\in[0,T]\times L^\i_{\cF_t}(\Om;\dbR^n),$$
we see that $P(t)\ges P_\e(t)$ for all $t\in[0,T]$ and hence
$$ R+D^\top P D \ges R +D^\top P_\e D \ges \e I_m, \q\ae~\hb{on}~[0,T],~\as $$
The property \rf{6-15:R+DPD>} therefore follows since $\e\in(0,\d)$ is arbitrary.
\end{proof}

In order to prove \autoref{thm:Riccati}, we also need the following lemma concerning the trace
of the product of two symmetric matrices; it is a special case of von Neumann's trace theorem
(see Horn and Johnson \cite[Theorem 7.4.1.1, page 458]{Horn-Johnson 2012}).

\begin{lemma}\label{lmm:algera} \sl
Let $\BA,\BB\in\dbS^n$ with $\BB$ being positive semi-definite. Then with $\l_{\max}(\BA)$ denoting
the largest eigenvalue of $\BA$, we have
$$\tr(\BA\BB)\les \l_{\max}(\BA)\cd\tr(\BB).$$
\end{lemma}

{\it Proof of \autoref{thm:Riccati}.}
We have seen from \autoref{lmm:pre-Ric} that the bounded process $P$ in \rf{V-property-1*} and the
processes defined by \rf{def-Pi-L} satisfy the BSDE \rf{6-15:dP=} and the relation \rf{6-16:relation}.
Further, \autoref{lmm:R+DPD>I} shows that
$$ R+D^\top PD \ges \d I_m, \q\ae~\hb{on}~[0,T],~\as $$
This, together with \rf{6-16:relation}, implies that
$$ \Th=-(R+D^\top PD)^{-1}(B^\top P +D^\top PC +D^\top\L +S), $$
which, substituted into \rf{6-15:dP=} yields \rf{Ric}.
It remains to prove that the process $\L$ is square-integrable. Set
\begin{align*}
\Si &= PA+A^\top P +C^\top PC +\L C+C^\top\L +Q, \\
\G  &= (PB +C^\top PD +\L D +S^\top)(R+D^\top PD)^{-1}(B^\top P +D^\top PC +D^\top\L +S).
\end{align*}
Because the matrix-valued processes $A,C,Q,P$ are all bounded and the process $\G$ is positive semi-definite,
we can choose a constant $K>0$ such that
\bel{Estimates}\left\{\begin{aligned}
& \tr[P(s)]+|P(s)|^2\les K,\\
& \tr[\Si(s)]\les K[1+|\L(s)|],\\
& \tr[P(s)\Si(s)]\les|P(s)||\Si(s)|\les K[1+|\L(s)|],\\
& \tr[-P(s)\G(s)]\les\l_{\max}[-P(s)]\tr[\G(s)]\les K\tr[\G(s)],
\end{aligned}\right.\ee
for Lebesgue-almost every $s$, $\dbP$-a.s. In the last inequality we have used \autoref{lmm:algera}.
In the sequel, we shall use the same letter $K$ to denote a generic positive constant whose value
might change from line to line.
Define for each $k\ges1$ the stopping time (with the convention $\inf\varnothing=\i$)
$$ \l_k = \inf\lt\{t\in[0,T];~\int_0^t|\L(s)|^2ds\ges k\rt\}.$$
Because the processes $\BU=\{\BU(s);0\les s\les T\}$ and $\BZ=\{\BZ(s);0\les s\les T\}$ are square-integrable,
$\BX^{-1}=\{\BX(s)^{-1};0\les s\les T\}$ is continuous, and $P,C,D$ are bounded,
we see from the definition \rf{def-Pi-L} of $\L$ that
$$\int_0^T|\L(s)|^2ds<\i, \q\as$$
This implies that $\lim_{k\to\i}\l_k=T$ almost surely. Since $P$ satisfies the SDE
$$dP(t)=\big[-\Si(t)+\G(t)\big]dt+\L(t)dW(t),$$
We have
\bel{6-19:P(s)=} P(t\wedge\l_k) = P(0) + \int_0^{t\wedge\l_k} \big[-\Si(s)+\G(s)\big] ds
                                       + \int_0^{t\wedge\l_k} \L(s)dW(s). \ee
Thanks to the definition of $\l_k$, the process
$$\lt\{\int_0^{t\wedge\l_k} \L(s)dW(s),\cF_t;0\les t\les T\rt\}
  =\lt\{\int_0^t \L(s){\bf1}_{\{s\les\l_k\}}dW(s),\cF_t;0\les t\les T\rt\} $$
is easily seen to be a matrix of square-integrable martingales, so taking expectations in \rf{6-19:P(s)=} gives
$$\dbE[P(t\wedge\l_k)]=P(0)+\dbE\int_0^{t\wedge\l_k} [-\Si(s)+\G(s)] ds. $$
This, together with \rf{Estimates}, implies that
\begin{align}\label{estimate-Gamma}
\dbE\int_0^{t\wedge\l_k} \tr[\G(s)]ds
   &=\dbE\,\tr[P(t\wedge\l_k)-P(0)] + \dbE\int_0^{t\wedge\l_k} \tr[\Si(s)]ds\les K\[1+\dbE\int_0^{t\wedge\l_k}|\L(s)|ds\].
\end{align}
On the other hand, we have by It\^{o}'s formula that
$$ d[P(t)]^2 = \big[P(-\Si+\G)+(-\Si+\G)P +\L^2\big] dt + (P\L+\L P)dW(t). $$
A similar argument based the definition of $\l_k$ shows that
$$\dbE [P(t\wedge\l_k)]^2=P(0)^2 + \dbE\int_0^{t\wedge\l_k}\Big\{P(s)[-\Si(s)+\G(s)]
                            +[-\Si(s)+\G(s)]P(s) +[\L(s)]^2\Big\} ds, $$
which, together with \rf{Estimates} and \rf{estimate-Gamma}, yields (recalling the Frobenius norm)
\begin{align}\label{estimate-G2}
& \dbE\int_0^{t\wedge\l_k} |\L(s)|^2ds = \tr\1n\[\dbE\int_0^{t\wedge\l_k} [\L(s)]^2ds\] \nn\\
&\q = \dbE|P(t\wedge\l_k)|^2-|P(0)|^2 + 2\dbE\int_0^{t\wedge\l_k}\tr[P(s)\Si(s)]ds
      +2\dbE\int_0^{t\wedge\l_k}\tr[-P(s)\G(s)]ds \nn\\
&\q \les K+K\dbE\int_0^{t\wedge\l_k}\[1+|\L(s)|\]ds +K\dbE\int_0^{t\wedge\l_k}\tr[\G(s)]ds\les K\[1+\dbE\int_0^{t\wedge\l_k}|\L(s)|ds\].
\end{align}
Furthermore, by the Cauchy-Schwarz inequality,
$$K\dbE\int_0^{t\wedge\l_k}|\L(s)|ds \les 2K^2 +{1\over2}\dbE\int_0^{t\wedge\l_k}|\L(s)|^2ds.$$
Combining this with \rf{estimate-G2} gives
$${1\over2}\dbE\int_0^{t\wedge\l_k} |\L(s)|^2ds \les K+2K^2. $$
Since the constant $K$ does not depend on $k$ and $t$, and $\lim_{k\to\i}\l_k=T$ almost surely,
we conclude that the process $\L$ is square-integrable by letting $k\to\i$ and then $t\uparrow T$.
$\hfill\qed$

\ms

We conclude this section with a proof of \autoref{thm:main}.

\medskip

{\it Proof of \autoref{thm:main}.} Suppose that \rf{convexity:uni} holds. Then Problem (SLQ) is uniquely solvable at any initial time $t<T$ according to \autoref{prop:t-tao}. In order to find the optimal control at any initial pair $(t,\xi)\in\cD$, it suffices to determine the optimal control $u_i=\{u_i(s);0\les s\les T\}$ at $(0,e_i)$ for each $i=1,\ldots,n$, since by \autoref{prop:ux=Ux} the optimal control $u^*_{t,\xi}$ at $(t,\xi)$ must be given by
$$u^*_{t,\xi}(s)=(u_1(s),\cdots,u_n(s))\xi; \q t\les s\les T.$$
With the notation \rf{7-29:notation}, we see from \autoref{thm:BX-invertible} that the process
$\BX=\{\BX(s);0\les s\les T\}$ is invertible. Therefore, finding the optimal controls $u_1,\cdots,u_n$
is equivalent to finding
$$\Th(s)=\BU(s)\BX(s)^{-1}; \q 0\les s\les T.$$
The latter can be accomplished by solving the SRE \rf{Ric}, whose solvability
is guaranteed by \autoref{thm:Riccati}. In fact, from the proof of \autoref{thm:Riccati} we can see that
$\Th$ is actually given by \rf{7-29:Th}.
Summarizing the above, we obtain the desired result.
$\hfill\qed$

\section{The Uniform Convexity of the Cost Functional}\label{sec:uniform-convexity}

In this section, we would like to present some sufficient conditions on the coefficients of the state equation and the weighting matrices of the cost functional that guarantee \rf{uni-convex*}. We first present the following result.

\begin{proposition} \sl Let {\rm\ref{(A1)}--\ref{(A2)}} hold. Then the mapping $u\mapsto J(t,0;u)$ is uniformly convex
for every $t\in[0,T)$ if either \rf{classical} or \rf{n=m} holds.
\end{proposition}

\begin{proof}
In the case that $S\equiv0$, we have
$$J(t,0;u) = \dbE\bigg\{\blan GX^{(u)}(T),X^{(u)}(T)\bran
             +\1n\int_t^T\[\blan Q(s)X^{(u)}(s),X^{(u)}(s)\bran+\lan R(s)u(s),u(s)\ran\] ds\bigg\},$$
where $X^{(u)}$ is the solution to the SDE
$$\left\{\begin{aligned}
   dX(s) &=[A(s)X(s)+B(s)u(s)]ds+[C(s)X(s)+D(s)u(s)]dW(s), \q s\in[t,T],\\
    X(t) &=0.
\end{aligned}\right.$$
If, in addition, condition \rf{classical} holds, then
$$ J(t,0;u) \ges \dbE\int_t^T\lan R(s)u(s),u(s)\ran ds \ges\d\,\dbE\int_t^T|u(s)|^2 ds,$$
which shows that the mapping $u\mapsto J(t,0;u)$ is uniformly convex.

\ms

Now, if, in addition to $S\equiv0$, condition \rf{n=m} holds, then
$$J(t,0;u) \ges \dbE\lan GX^{(u)}(T),X^{(u)}(T)\ran \ges \d\,\dbE|X^{(u)}(T)|^2.$$
Since $D(s)^\top D(s)\ges\d I_m$, $[D(s)^\top D(s)]^{-1}$ exists and is uniformly bounded.
Therefore, by \autoref{lmm:sol-F&B-SDE} the BSDE
\bel{Y:7-18}\left\{\begin{aligned}
   dY(s) &= \Big\{\big[A-B(D^\top D)^{-1}D^\top C\big]Y + B(D^\top D)^{-1}D^\top Z\Big\}ds + ZdW(s), \q s\in[t,T],\\
    Y(T) &= X^{(u)}(T)
\end{aligned}\right.\ee
admits a unique adapted solution $(Y,Z)\in L_\dbF^2(\Om;C([t,T];\dbR^n))\times L^2_\dbF(t,T;\dbR^n)$ satisfying
\bel{8-11:ineq} \dbE\lt[ \sup_{t\les s\les T}|Y(s)|^2 + \int_t^T|Z(s)|^2ds \rt]\les K\dbE|X^{(u)}(T)|^2 \ee
for some constant $K>0$ independent of $X^{(u)}(T)$. On the other hand, it is easy to verify that the adapted solution of \rf{Y:7-18} is given by
$$Y(s)=X^{(u)}(s), \q Z(s)=C(s)X^{(u)}(s)+D(s)u(s); \q s\in[t,T].$$
Note that for any $\a,\b\in\dbR$ and any $\e>0$,
$$\a^2=(\a+\b-\b)^2=(\a+\b)^2-2(\a+\b)\b+\b^2\les{1+\e\over\e}(\a+\b)^2
+(1+\e)\b^2,$$
which leads to
$$(\a+\b)^2\ges{\e\over1+\e}\a^2-\e\b^2.$$
Thus, using the condition $D(s)^\top D(s)\ges\d I_m$, we have
\begin{align*}
\dbE\int_t^T|Z(s)|^2ds
   &\ges{\e\over1+\e}\,\dbE\int_t^T|D(s)u(s)|^2ds -\e\,\dbE\int_t^T|C(s)X^{(u)}(s)|^2ds \\
   &\ges{\e\d\over1+\e}\,\dbE\int_t^T|u(s)|^2ds-\e\|C(\cd)\|_\i^2T\,\dbE\lt[\sup_{t\les s\les T}|X^{(u)}(s)|^2\rt],
\end{align*}
where
$$\|C(\cd)\|_\infty=\esssup_{(s,\om)\in[t,T]\times\Om}|C(s,\om)|.$$
It follows from \rf{8-11:ineq} that
\begin{align*}
K\dbE|X^{(u)}(T)|^2
   &\ges \dbE\lt[ \sup_{t\les s\les T}|X^{(u)}(s)|^2 + \int_t^T|Z(s)|^2ds \rt]\\
   &\ges{\e\d\over1+\e}\,\dbE\int_t^T|u(s)|^2ds + \big[1-\e\|C(\cd)\|_\infty^2T\big]\dbE\lt[\sup_{t\les s\les T}|X^{(u)}(s)|^2\rt] \\
   &\ges{\e\d\over1+\e}\,\dbE\int_t^T|u(s)|^2ds,
\end{align*}
provided $0<\e\les{1\over\|C(\cd)\|_\i^2T}$. Hence,
$$J(t,0;u)\ges \d\,\dbE|X^{(u)}(T)|^2 \ges{\e\d^2\over K(1+\e)}\dbE\int_t^T|u(s)|^2ds, \q\forall u\in\cU[t,T].$$
This completes the proof.  
\end{proof}

The above result shows that the cases discussed in \cite{Tang 2003,Kohlmann-Tang 2003} are special cases of the uniform convexity condition presented in this paper. The next result shows that there is a class of problems for which neither \rf{classical} nor \rf{n=m} holds,
but the mapping $u\mapsto J(t,0;u)$ is uniformly convex. Therefore, the case we are discussing in the current paper is strictly more general than those in \cite{Tang 2003,Kohlmann-Tang 2003}.

\medskip

Consider the case that $B(\cd)=0,C(\cd)=0,S(\cd)=0$. Let $\F=\{\F(s);0\les s\les T\}$ be the solution to the random ordinary differential equation
$$\left\{\begin{aligned}
   d\F(s) &= A(s)\F(s)ds, \q s\in[0,T], \\
    \F(0) &= I_n.
\end{aligned}\right.$$
Then $\F(s)$ is invertible for every $s\in[0,T]$, and since $A$ is a bounded process, both $\F(s)$ and $\F(s)^{-1}$ are bounded $\cF_s$-adapted matrix-valued random variables. Denote
$$\|\F(s)\|_\i=\esssup_{\om\in\Om}|\F(s,\om)|,\qq\|\F(s)^{-1}\|_\i
=\esssup_{\om\in\Om}|\F(s,\om)^{-1}|.$$
Furthermore, let $\l_G$ and $\l_{Q(s)}$ be the essential infimums
of the smallest eigenvalues of $G$ and $Q(s)$, respectively. Hence,
$$G\ges\l_GI_n,\qq Q(s)\ges\l_{Q(s)}I_n, \q\as,~\ae~s\in[0,T].$$

\begin{theorem}\label{thm:uniconvex-cndtn}
\sl Let {\rm\ref{(A1)}--\ref{(A2)}} hold. Suppose that $B(\cd)=0,C(\cd)=0,S(\cd)=0$. If
\begin{equation}\label{unitu-condition}
\left[{\l_G\over\|\F(T)^{-1}\|_\i^2}+\int_r^T{\l_{Q(s)}\over\|\F(s)^{-1}
\|_\i^2}ds\right]{1\over\|\F(r)\|_\i^2}
D(r)^\top D(r)+R(r)\ges\d I_m,\q\as,~\ae~r\in[0,T],
\end{equation}
for some $\d>0$, then the mapping $u\mapsto J(t,0;u)$ is uniformly convex for every $t\in[0,T)$.
\end{theorem}

Note that \rf{unitu-condition} allows $R(r)$ to be negative definite if $D^\top D$ is sufficiently positive definite, or, to be indefinite/partially negative definite (within a certain range) and $D^\top D$ is partially positive definite in an obvious sense. Therefore, it is possible that neither \rf{classical} nor \rf{n=m} holds.

\begin{proof}
Let $t\in[0,T)$ be fixed. Since $B(\cd)=0$ and $C(\cd)=0$, for each $u\in\cU[t,T]$,
the solution of the state equation \rf{state} with initial state $\xi=0$ is given by
$$X(s)=\F(s)\int_t^s\F(r)^{-1}D(r)u(r)dW(r),\q s\in[t,T].$$
For any $(n\times m)$ matrix $F$, from the inequalities
\begin{align*}
&|F| = |\F(s)^{-1}\F(s)F| \les |\F(s)^{-1}||\F(s)F| \les \|\F(s)^{-1}\|_\i|\F(s)F|, \\
&|F| = |\F(s)\F(s)^{-1}F| \les |\F(s)||\F(s)^{-1}F| \les \|\F(s)\|_\i|\F(s)^{-1}F|,
\end{align*}
we have
$$|\F(s)F|\ges{1\over\|\F(s)^{-1}\|_\i}|F|, \q |\F(s)^{-1}F|\ges {1\over\|\F(s)\|_\i}|F|.$$
Thus,
\begin{align*}
\dbE\lan GX(T),X(T)\ran
  &\ges \dbE\left[\l_G\bigg|\F(T)\int_t^T \F(r)^{-1}D(r)u(r)dW(r)\bigg|^2\right] \\
  &\ges \dbE\left[{\l_G\over\|\F(T)^{-1}\|_\i^2}\bigg|\int_t^T \F(r)^{-1}D(r)u(r)dW(r)\bigg|^2\right] \\
  &=    {\l_G\over\|\F(T)^{-1}\|_\i^2}\dbE\int_t^T |\F(r)^{-1}D(r)u(r)|^2 dr \\
  &\ges {\l_G\over\|\F(T)^{-1}\|_\i^2}\dbE\int_t^T {1\over\|\F(r)\|_\i^2}|D(r)u(r)|^2 dr,
\end{align*}
and similarly,
$$\dbE\lan Q(s)X(s),X(s)\ran \ges {\l_{Q(s)}\over\|\F(s)^{-1}\|_\i^2}\dbE\int_t^s {1\over\|\F(r)\|_\i^2}|D(r)u(r)|^2dr. $$
Using Fubini's theorem we obtain
$$\dbE\int_t^T\lan Q(s)X(s),X(s)\ran ds
\ges \dbE\int_t^T\left[\int_r^T{\l_{Q(s)}\over\|\F(s)^{-1}\|_\i^2}ds\right]{1\over\|\F(r)\|_\i^2}|D(r)u(r)|^2dr. $$
Therefore, denoting
$$H(r)= \left[{\l_G\over\|\F(T)^{-1}\|_\i^2}+\int_r^T{\l_{Q(s)}\over\|\F(s)^{-1}\|_\i^2}ds
\right]{1\over\|\F(r)\|_\i^2}D(r)^\top D(r)+R(r),$$
we have
\begin{align*}
J(t,0;u) &= \dbE\lan GX(T),X(T)\ran + \dbE\int_t^T\[\lan Q(s)X(s),X(s)\ran +\lan R(s)u(s),u(s)\ran\]ds \\
         &\ges \dbE\int_t^T \lan H(r)u(r),u(r)\ran dr.
\end{align*}
So the mapping $u\mapsto J(t,0;u)$ is uniformly convex when \rf{unitu-condition} holds for some $\d>0$.
\end{proof}

Although the above result gives a class of problems for which neither \rf{classical} nor \rf{n=m} holds, and the mapping $u\mapsto J(t,0;u)$ is uniformly convex, the imposed conditions seem to be a little too restrictive. In the rest of this section, we would like to explore the problem a little more.

\ms

Note that \autoref{thm:main} can be read as follows:
Under {\rm\ref{(A1)}--\ref{(A2)}}, if $u\mapsto J(0,0;u)$ is uniformly convex,
then there exists an $\dbF$-adapted $\dbS^n$-valued process $P$ such that
\bel{R+DPD>0} R(s)+D(s)^\top P(s)D(s)\ges\l I_m, \q \ae~s\in[0,T],~\as\ee
for some $\l>0$. From this, we see that the mapping $u\mapsto J(0,0;u)$ could never be uniformly convex if
\bel{R<0}R(s)\les0,\q D(s)=0, \q \ae~s\in[0,T],~\as\ee
Thus, a natural necessary condition for $u\mapsto J(0,0;u)$ to be uniformly convex is that \rf{R<0} fails.
Now, we provide the following general sufficient condition for the uniform convexity of $u\mapsto J(t,0;u)$.

\begin{theorem}\sl
Let {\rm\ref{(A1)}--\ref{(A2)}} hold. Let $t\in[0,T)$ and $Q_0\in L^\i_\dbF(t,T;\dbS^n)$ with
$$Q_0(s)>0, \q\ae~s\in[t,T],~\as$$
Let $(\Pi,\Si)\in L_\dbF^\i(\Om;C([t,T];\dbS^n))\times L_\dbF^2(t,T;\dbS^n)$ be the adapted solution
to the following {\it Lyapunov BSDE}:
\bel{Lyapunov}\left\{\begin{aligned}
d\Pi(s) &= -(\Pi A+A^\top\Pi+C^\top \Pi C+\Si C+C^\top\Si+Q-Q_0)ds+\Si dW(s),\q s\in[t,T],\\
 \Pi(T) &= G.
\end{aligned}\right.\ee
If for some $\d>0$,
\bel{R+DPD>0*}\begin{aligned}
R+D^\top\Pi D -(B^\top\Pi+D^\top \Pi C+D^\top\Si+S)Q_0^{-1}(\Pi B+C^\top\Pi D+\Si D+S^\top)\ges\d I_m,\\
\ae~\hb{on}~[t,T],~\as
\end{aligned}\ee
then $u\mapsto J(t,0;u)$ is uniformly convex.
\end{theorem}

\begin{proof}
For any bounded $u\in\cU[t,T]$, let $X_0=\{X_0(s);t\les s\les T\}$ be the state process
corresponding to $u$ and the initial state $\xi=0$. Denote
$$\G=-(\Pi A+A^\top\Pi+C^\top\Pi C+\Si C+C^\top\Si+Q-Q_0),$$
with $(\Pi,\Si)$ being the adapted solution to \rf{Lyapunov}.
By It\^o's formula, we have
$$d(\Pi X_0)=(\G X_0+\Pi AX_0+\Pi Bu+\Si CX_0+\Si Du)ds+(\Si X_0+\Pi CX_0+\Pi Du)dW,$$
and hence
\begin{align*}
d\lan \Pi X_0,X_0\ran &= \[\lan\G X_0+\Pi AX_0+\Pi Bu+\Si CX_0+\Si Du,X_0\ran+\lan\Pi X_0,AX_0+Bu\ran\\
&\hp{=\ } +\lan\Si X_0+\Pi CX_0+\Pi Du,CX_0+Du\ran\]ds\\
&\hp{=\ } +\[\lan\Si X_0+\Pi CX_0+\Pi Du,X_0\ran+\lan\Pi X_0,CX_0+Du\ran\]dW\\
&= \[\lan(\G+\Pi A+A^\top\Pi+C^\top \Pi C+\Si C+C^\top\Si)X_0,X_0\ran\\
&\hp{=\ } +2\lan(B^\top\Pi+D^\top\Pi C+D^\top\Si)X_0,u\ran+\lan D^\top \Pi Du,u\ran\]ds\\
&\hp{=\ } +\[\lan(\Si+\Pi C+C^\top\Pi)X_0,X_0\ran+2\lan D^\top \Pi X_0,u\ran\]dW\\
&= \[\lan(Q_0-Q)X_0,X_0\ran+2\lan(B^\top\Pi+D^\top \Pi C+D^\top\Si)X_0,u\ran+\lan D^\top\Pi Du,u\ran\]ds\\
&\hp{=\ } +\[\lan(\Si+\Pi C+C^\top\Pi)X_0,X_0\ran+2\lan D^\top \Pi X_0,u\ran\]dW.
\end{align*}
Taking expectations on both sides (possibly together with a localization argument) gives
\begin{align*}
\dbE\lan GX_0(T),X_0(T)\ran &= \dbE\int_t^T\[\lan(Q_0-Q)X_0,X_0\ran+2\lan(B^\top \Pi+D^\top\Pi C+D^\top\Si)X_0,u\ran\\
                            &\hp{= \dbE\int_t^T\[} +\lan D^\top\Pi Du,u\ran\]ds.
\end{align*}
Substituting the above into the cost functional, we obtain
\begin{align*}
J(t,0;u) &= \dbE\int_t^T\[\lan Q_0X_0,X_0\ran+2\lan(B^\top \Pi+D^\top\Pi C+D^\top\Si+S)X_0,u\ran+\lan(R+D^\top\Pi D)u,u\ran\]ds\\
&= \dbE\int_t^T\Big\{\big|Q_0^{1\over2}X_0+Q_0^{-{1\over2}}(\Pi B+C^\top\Pi D+\Si D+S^\top)u\big|^2\\
&\hp{=\ } +\blan\big[R+D^\top\Pi D-(B^\top\Pi+D^\top \Pi C+D^\top\Si+S)Q_0^{-1}(\Pi B+C^\top\Pi D+\Si D+S^\top)\big]u,u\bran\Big\}ds\\
&\ges\d\dbE\int_t^T|u(s)|^2ds.
\end{align*}
This proves our conclusion for bounded $u\in\cU[t,T]$.
The unbounded case follows immediately since bounded controls are dense in $\cU[t,T]$.
\end{proof}

The above result gives some compatibility conditions among the coefficients of the state equation
and the weighting matrices in the cost functional that ensure the uniform convexity of the cost
functional in $u$. Let us look at several special cases.
\begin{enumerate}[(i)]
\item Let $\l>0$ and $Q_0=\l I_n$. Then with $(\Pi_\l,\Si_\l)$ denoting the adapted solution
to the following Lyapunov BSDE:
$$\left\{\begin{aligned}
d\Pi_\l(s) &= -(\Pi_\l A+A^\top\Pi_\l+C^\top\Pi_\l C+\Si_\l C+C^\top\Si_\l+Q-\l I_n)ds+\Si_\l dW(s),\q s\in[t,T],\\
 \Pi_\l(T) &= G,
\end{aligned}\right.$$
the corresponding condition \rf{R+DPD>0*} reads
\begin{align*}
R+D^\top\Pi_\l D- \l^{-1}(B^\top\Pi_\l+D^\top\Pi_\l C+D^\top\Si_\l+S)(\Pi_\l B+C^\top\Pi_\l D+\Si_\l D+S^\top)\ges\d I_m,\\
\ae~s\in[t,T],~\as
\end{align*}

\item Let all the coefficients and weighting matrices be deterministic.
Then, $\Si\equiv0$ and \rf{Lyapunov} reads
$$\left\{\begin{aligned}
& \dot\Pi+\Pi A+A^\top\Pi+C^\top\Pi C+Q-Q_0=0,\q s\in[t,T],\\
& \Pi(T)=G,
\end{aligned}\right.$$
and condition \rf{R+DPD>0*} becomes
$$ R+D^\top\Pi D-(B^\top\Pi+D^\top \Pi C+S)Q_0^{-1}(\Pi B+C^\top\Pi D+S^\top)\ges\d I_m,\q\ae~s\in[t,T]. $$
This is new even for the deterministic case previously studied in the literature.
Further, with $Q_0=\l I_n$, the above become
$$\left\{\begin{aligned}
& \dot\Pi_\l+\Pi_\l A+A^\top\Pi_\l+C^\top\Pi_\l C+Q-\l I_n=0,\q s\in[t,T],\\
& \Pi_\l(T)=G,
\end{aligned}\right.$$
and
$$ R+D^\top\Pi_\l D- \l^{-1}(B^\top\Pi_\l+D^\top\Pi_\l C+S)(\Pi_\l B+C^\top\Pi_\l D+S^\top)\ges\d I_m,\q\ae~s\in[t,T]. $$

\item The coefficients are still random and $B=0$, $C=0$, $S=0$. Then \rf{Lyapunov} becomes
$$\left\{\begin{aligned}
d\Pi(s) &= -(\Pi A+A^\top\Pi+Q-Q_0)ds+\Si dW(s),\q s\in[t,T],\\
 \Pi(T) &= G,
\end{aligned}\right.$$
and \rf{R+DPD>0*} reads
$$ R+D^\top(\Pi-\Si Q_0^{-1}\Si) D\ges\d I_m,\q\ae~s\in[t,T],~\as $$
This is comparable with the result of \autoref{thm:uniconvex-cndtn}.
\end{enumerate}

\section{An Illustrative Example}\label{sec:example}

In this section, we present an illustrative example for which the cost functional is uniformly convex
and the associated stochastic Riccati equation admits a unique adapted solution $(P,\L)$ with $P$ being
not positive definite and with $\L$ being unbounded.

\begin{example}\rm

Let $\eta\in L^\i_{\cF_T}(\Om;\dbR)$ be a Malliavin differentiable random variable with
square-integrable Malliavin derivative $D_t\eta$. Then the Clark--Ocone formula implies that
$$\eta = \dbE\eta + \int_0^T\dbE[D_t\eta|\cF_t]dW(t).$$
Let $\m(t)$ be a right-continuous modification of $\dbE[\eta|\cF_t]$, and let
$\l(t)$ be a right-continuous modification of $\dbE[D_t\eta|\cF_t]$. Then
\bel{19-7-19} \m(t) = \dbE\eta + \int_0^t\l(s)dW(s), \q t\in[0,T].\ee
Consider the SLQ problem where the coefficients of the state equation are given by
$$ A(s)=\begin{pmatrix}0&1\\0&0\end{pmatrix},
\q B(s)=\begin{pmatrix}0\\0\end{pmatrix},
\q C(s)=\begin{pmatrix}0&0\\0&0\end{pmatrix},
\q D(s)=\begin{pmatrix}1\\2\end{pmatrix},$$
and the weighting matrices in the cost functional are given by
\begin{alignat*}{2}
& G = \begin{pmatrix}-(1+T^2)&T\\ T&1+T^2\end{pmatrix} + \eta\begin{pmatrix}4&-2\\-2&1\end{pmatrix},  \q~&& S(s) = (0,0),\\
& Q(s) = \begin{pmatrix}2s&s^2\\ s^2&-4s\end{pmatrix} + 4\m(s)\begin{pmatrix}0&-1\\-1&1\end{pmatrix}, \q~&& R(s)=-(1+s^2).
\end{alignat*}
For this SLQ problem, the associated stochastic Riccati equation reads
\bel{ex:P-eqn}\left\{\begin{aligned}
 dP(t) &= -\big[PA+A^\top P+Q-\L D(R+D^\top PD)^{-1}D^\top\L\big]dt+\L dW(t),\q t\in[0,T],\\
  P(T) &= G.
\end{aligned}\right.\ee

\ss

We claim that the adapted solution $(P,\L)$ of \rf{ex:P-eqn} is given by the following:
\begin{equation}\label{PL}
 P(t) = \begin{pmatrix}-(1+t^2)&t \\ t&1+t^2\end{pmatrix} + \m(t)\begin{pmatrix}4&-2\\-2&1\end{pmatrix}, \qq
\L(t) = \l(t)\begin{pmatrix}4&-2\\-2&1\end{pmatrix}.
\end{equation}
In fact, we have from \rf{19-7-19} and \rf{PL} that on the one hand
\begin{equation}\label{ex:dP}
dP(t) = \begin{pmatrix}-2t&1\\1&2t\end{pmatrix}dt + \l(t)\begin{pmatrix}4&-2\\-2&1\end{pmatrix}dW(t)
      = \begin{pmatrix}-2t&1\\1&2t\end{pmatrix}dt + \L(t)dW(t); \q
 P(T) = G.
\end{equation}
On the other hand, a direct computation shows that
\begin{align*}
 P(t)A(t)+A(t)^\top P(t)+Q(t) = \begin{pmatrix}2t&-1\\-1&-2t\end{pmatrix}, \q
           D(t)^\top P(t)D(t) = 3(1+t^2)+4t, \q
                    \L(t)D(t) = 0,
\end{align*}
which leads to
\begin{align}\label{R+DPD>2}
& R(t)+D(t)^\top P(t)D(t)=2(1+t^2)+4t\ges2, \\
& \L(t)D(t)[R(t)+D(t)^\top P(t)D(t)]^{-1}D(t)^\top\L(t)=0. \nn
\end{align}
Consequently,
\begin{equation}\label{ex:coe}
-\big[PA+A^\top P+Q-\L D(R+D^\top PD)^{-1}D^\top\L\big] = \begin{pmatrix}-2t&1\\1&2t\end{pmatrix}.
\end{equation}
Combining \rf{ex:dP} and \rf{ex:coe}, we see that $(P,\L)$ is the adapted solution to the Riccati
equation \rf{ex:P-eqn}.

\ms

From \rf{PL}, we see that if $\eta$ can be chosen so that
\bel{ex:eta-cndtn} 0<\esssup_{\om\in\Om}|\eta(\om)|< {1\over4} \q\hb{and}\q D_t\eta \hb{ is unbounded},\ee
then $\l(t)=\dbE[D_t\eta|\cF_t]$, and hence $\L(t)$, is an unbounded process, and
$$P(t) = \begin{pmatrix}4\m(t)-1-t^2&t-2\m(t) \\ t-2\m(t)&\m(t)+1+t^2\end{pmatrix}$$
is not positive definite.
There are many random variables that satisfy \rf{ex:eta-cndtn}. For example, we can take
$$\eta={1\over8}\sin[W(T)^2].$$
The Malliavin derivative of this $\eta$ is
$$D_t\eta={1\over4}W(T)\cos[W(T)^2],$$
which is clearly unbounded.

\ms

The cost functional of this SLQ problem is uniformly convex.
This can be seen by applying It\^o's formula to $s\mapsto\lan P(s)X(s),X(s)\ran$,
where $X$ is the solution to the state equation \rf{state} with initial pair $(0,0)$,
which, in our situation, reads
$$\left\{\begin{aligned}
   dX(s) &= A(s)X(s)ds + D(s)u(s)dW(s), \q s\in[0,T],\\
    X(0) &= 0.
\end{aligned}\right.$$
More precisely, by noting that $\L(t)D(t) =0$ and using It\^o's formula, we have
\begin{align*}
d\lan P(s)X(s),X(s)\ran &= [-\lan Q(s)X(s),X(s)\ran + \lan P(s)D(s)u(s),D(s)u(s)\ran]ds \\
                        &\hp{=\ } + \lan 2P(s)D(s)u(s)+\L(s)X(s),X(s)\ran dW(s).
\end{align*}
Thus, taking expectations on both sides (together with a localization argument) gives
$$\dbE\lan GX(T),X(T)\ran = \dbE\int_0^T[-\lan Q(s)X(s),X(s)\ran + \lan P(s)D(s)u(s),D(s)u(s)\ran]ds.$$
Substituting the above into the cost functional and using \rf{R+DPD>2}, we obtain
$$J(0,0;u) = \dbE\int_0^T\blan R(s)+D(s)^\top P(s)D(s)u(s),u(s)\bran ds \ges 2\dbE\int_0^T|u(s)|^2 ds.$$
This shows that the cost functional of this SLQ problem is uniformly convex.
\end{example}

\section{Concluding Remarks}\label{sec:remarks}

In this paper, for a stochastic linear-quadratic optimal control problem with random coefficients in which
the weighting matrices of the cost functional are allowed to be indefinite,
we showed that under the uniform convexity condition on the cost functional, the stochastic Riccati equation
admits a unique adapted solution which can be constructed by the open-loop optimal pair, together with its
adjoint equation.
Moreover, the open-loop optimal control admits a state feedback/closed-loop representation.
For simplicity, the Brownian motion under consideration is assumed to be one-dimensional.
In the case of a $d$-dimensional Brownian motion $W=\{(W_1(t),\ldots,W_d(t));0\les t<\i\}$,
the SLQ optimal control problem is to find a control $u^*\in\cU[t,T]$ such that the quadratic cost functional
\begin{equation*}
J(t,\xi;u) =\dbE\left[\lan GX(T),X(T)\ran
                +\int_t^T\llan\1n\begin{pmatrix}Q(s)&\1nS(s)^\top \\ S(s)&\1nR(s)\end{pmatrix}\1n
                                 \begin{pmatrix}X(s) \\ u(s)\end{pmatrix}\1n,
                                 \begin{pmatrix}X(s) \\ u(s)\end{pmatrix}\1n\rran ds\right]
\end{equation*}
is minimized subject to the following state equation:
$$\left\{\begin{aligned}
  dX(s) &= [A(s)X(s)+B(s)u(s)]ds +{\ds\sum^d_{i=1}}[C_i(s)X(s)+D_i(s)u(s)]dW_i(s), \q s\in[t,T],\\
   X(t) &= \xi,
\end{aligned}\right.$$
where the weighting matrices in the cost functional satisfy \ref{(A2)},
and the coefficients of the state equation satisfy the following assumption that is similar to \ref{(A1)}:
\begin{itemize}[leftmargin=1.07cm]
\item[{\bf\setword{(A1)$^\prime$}{(A1)*}}] The processes $A,C_i:[0,T]\times\Om\to\dbR^{n\times n}$ and
$B,D_i:[0,T]\times\Om\to\dbR^{n\times m}$ $(i=1,\ldots,d)$ are bounded and $\dbF$-progressively measurable.
\end{itemize}
In this case, the associated SRE becomes
\bel{Ric-d}\left\{\begin{aligned}
dP(t) &=-\Big\{PA+A^\top P +{\ds\sum^d_{i=1}}(C_i^\top PC_i +\L_i C_i+C_i^\top\L_i) +Q \\
      &\hp{=-\[} -\[PB +{\ds\sum^d_{i=1}}\big(C_i^\top P +\L_i\big)D_i +S^\top\]
                  \(R+{\ds\sum^d_{i=1}}D_i^\top PD_i\)^{-1} \\
      &\hp{=-\[} \times\[B^\top P +{\ds\sum^d_{i=1}}D_i^\top\big(PC_i +\L_i\big) +S\]\Big\}dt
                 +{\ds\sum^d_{i=1}}\L_i dW_i(t), \q t\in[0,T],\\
 P(T) &=G,
\end{aligned}\right.\ee
and the corresponding main result \autoref{thm:main} can be stated as follows.

\begin{theorem} \sl
Let {\rm\ref{(A1)*}} and {\rm\ref{(A2)}} hold. Suppose that there exists a constant $\d>0$ such that
$$J(0,0;u) \ges \d\,\dbE\int_0^T|u(s)|^2ds, \q\forall u\in \cU[0,T]. $$
Then Problem {\rm(SLQ)} is uniquely solvable and the SRE \rf{Ric-d} admits a unique adapted solution
$(P,\L)=(P,\L_1,\ldots,\L_d)$ such that
$$ R+{\ds\sum^d_{i=1}}D_i^\top PD_i \ges \l I_m, \q\ae~\hb{on}~[0,T],~\as $$
holds for some constant $\l>0$.
Moreover, the unique optimal control $u^*_{t,\xi}=\{u^*_{t,\xi}(s);t\les s\les T\}$ at $(t,\xi)\in\cS[0,T)\times L^\i_{\cF_t}(\Om;\dbR^n)$ admits the following linear state feedback representation:
$$u^*_{t,\xi}(s) = \Th(s)X^*(s); \q s\in[t,T],$$
where $\Th$ is defined by
$$ \Th = -\(R+{\ds\sum^d_{i=1}}D_i^\top PD_i\)^{-1}\[B^\top P +{\ds\sum^d_{i=1}}D_i^\top\big(PC_i+\L_i\big) +S\], $$
and $X^*=\{X^*(s);t\les s\les T\}$ is the solution the closed-loop system
$$\left\{\begin{aligned}
  dX^*(s) &=[A(s)+B(s)\Th(s)]X^*(s)ds +{\ds\sum^d_{i=1}}[C_i(s)+D_i(s)\Th(s)]X^*(s) dW_i(s), \q s\in[t,T],\\
   X^*(t) &=\xi.
\end{aligned}\right.$$
\end{theorem}

\bs

\bf Acknowledgement. \rm The authors would like to thank the anonymous referees for their critical and suggestive comments on the previous version of the paper.


\begin{thebibliography}{90}
\addtolength{\itemsep}{-1.0ex}

\bibitem{Rami-Moore-Zhou 2002} M.~Ait Rami, J.~B.~Moore, and X.~Y.~Zhou,
\it Indefinite stochastic linear quadratic control and generalized differential Riccati equation,
\sl SIAM J. Control Optim., \rm 40 (2002), 1296--1311.

\bibitem{Anderson-Moore 1989} B.~D.~O.~Anderson and J.~B.~Moore,
\sl Optimal Control: Linear Quadratic Methods,
\rm Prentice-Hall, Englewood Cliffs, NJ, 1989.

\bibitem{Athens 1971} M.~Athens,
\it Special issues on linear-quadratic-Gaussian problem,
\sl IEEE Trans. Automat. Contr., \rm AC-16 (1971), 527--869.

\bibitem{Bellman-Wing 1975} R.~E.~Bellman and G.~M.~Wing,
\sl An Introduction to Invariant Imbedding,
\rm Wiley, New York, 1975.

\bibitem{Bensoussan 1982} A.~Bensoussan,
\it Lectures on stochastic control,
\sl in Nonlinear Filtering and Stochastic Control, Lecture Notes in Math. 972, \rm Springer-Verlag, New York, 1982, 1--62.

\bibitem{Bismut 1976} J.~M.~Bismut,
\it Linear quadratic optimal stochastic control with random coefficients,
\sl SIAM J. Control Optim., \rm 14 (1976), 419--444.

\bibitem{Bismut 1978} J.~M.~Bismut,
\it Controle des systems lineares quadratiques: applications de l'integrale stochastique,
\sl in S\'{e}minaire de Probabilit\'{e}s XII, Lecture Notes in Math. 649, C. Dellacherie,
P. A. Meyer, and M. Weil, eds., Springer-Verlag, Berlin, \rm 1978, 180--264.

\bibitem{Chen-Li-Zhou 1998} S.~Chen, X.~Li, and X.~Y.~Zhou,
\it Stochastic linear quadratic regulators with indefinite control weight costs,
\sl SIAM J. Control Optim., \rm 36 (1998), 1685--1702.

\bibitem{Chen-Yong 2000} S.~Chen and J.~Yong,
\it Stochastic linear quadratic optimal control problems with random coefficients,
\sl Chin. Ann. Math. Ser. B, \rm 21 (2000), 323--338.

\bibitem{Chen-Yong 2001} S.~Chen and J.~Yong,
\it Stochastic linear quadratic optimal control problems,
\sl Appl. Math. Optim., \rm 43 (2001), 21--45.

\bibitem{Chen-Zhou 2000} S.~Chen and X.~Y.~Zhou,
\it Stochastic linear quadratic regulators with indefinite control weight costs. II,
\sl SIAM J. Control Optim., \rm 39 (2000), 1065--1081.

\bibitem{Davis 1977} M.~H.~A.~Davis,
\sl Linear Estimation and Stochastic Control,
\rm Chapman and Hall, London, 1977.

\bibitem{Horn-Johnson 2012} R.~A.~Horn and C.~R.~Johnson,
\sl Matrix Analysis, 2nd ed., \rm Cambridge, New York, 2012.

\bibitem{Hu-Zhou 2003} Y.~Hu and X.~Y.~Zhou,
\it Indefinite stochastic Riccati equations,
\sl SIAM J. Control Optim., \rm 42 (2003), 123--137.


\bibitem{Kohlmann-Tang 2003} M.~Kohlmann and S.~Tang,
\it Multidimensional backward stochastic Riccati equations and applications,
\sl SIAM J. Control Optim., \rm 41 (2003), 1696--1721.


\bibitem{Li-Wu-Yu 2018} N.~Li, Z.~Wu, and Z.~Yu,
\it Indefinite stochastic linear-quadratic optimal control problems with random jumps and related stochastic Riccati equations,
\sl Science China--Mathematics, \rm 61 (2018), 563--576.

\bibitem{Lim-Zhou 1999} A.~E.~B.~Lim and X.~Y.~Zhou,
\it Stochastic optimal LQR control with integral quadratic constraints and indefinite control weights,
\sl IEEE Trans. Automat. Contr., \rm 44 (1999), 359--369.

\bibitem{Lim-Zhou 2002} A.~E.~B.~Lim and X.~Y.~Zhou,
\it Mean-variance portfolio selection with random parameters in a complete market,
\sl Math. Oper. Res., \rm 27 (2002), 101--120.

\bibitem{Lu-Wang-Zhang 2017} Q.~L\"u, T.~Wang, and X.~Zhang,
\it Characterization of optimal feedback for stochastic linear quadratic control problems,
\sl Probab. Uncertain. Quant. Risk, \rm 2:11 (2017).

\bibitem{Molinari 1977} B.~P.~Molinari,
\it The time-invariant linear-quadratic optimal control probem,
\sl Automatica, \rm 13 (1977), 347--357.

\bibitem{Mou-Yong 2006} L.~Mou and J.~Yong,
\it Two-person zero-sum linear quadratic stochastic differential games by a Hilbert space method,
\sl J. Industrial \& Management Optim., \rm 2 (2006), 95--117.

\bibitem{Peng 1999} S.~Peng,
\it Open problems on backward stochastic differential equations,
\sl Control of Distributed Parameter and Stochastic Systems, S.~Chen, X.~Li, J.~Yong, and X.~Y.~Zhou, eds., \rm Springer, Boston, MA, 1999, 265--273.

\bibitem{Qian-Zhou 2013} Z.~Qian and X.~Y.~Zhou,
\it Existence of solutions to a class of indefinite stochastic Riccati equations,
\sl SIAM J. Control Optim., \rm 51 (2013), 221--229.

\bibitem{Sun-Li-Yong 2016} J.~Sun, X.~Li, and J.~Yong,
\it Open-loop and closed-loop solvabilities for stochastic linear quadratic optimal control problems,
\sl SIAM J. Control Optim., \rm 54 (2016), 2274--2308.

\bibitem{Sun-Yong 2014} J.~Sun and J.~Yong,
\it Linear quadratic stochastic differential games: Open-loop and closed-loop saddle points,
\sl SIAM J. Control Optim., \rm 52 (2014), 4082--4121.

\bibitem{Tang 2003} S.~Tang,
\it General linear quadratic optimal stochastic control problems with random coefficients:
    Linear stochastic Hamilton systems and backward stochastic Riccati equations,
\sl SIAM J. Control Optim., \rm 42 (2003), 53--75.

\bibitem{Tang 2016} S.~Tang,
\it Dynamic programming for general linear quadrtic optimal stochastic control with random coefficients,
\sl SIAM J. Control Optim., \rm 53 (2016), 1082--1106.

\bibitem{Wonham 1968} W.~M.~Wonham,
\it On a matrix Riccati equation of stochastic control,
\sl SIAM J. Control, \rm 6 (1968), 681--697.

\bibitem{Yong-Zhou 1999} J.~Yong and X.~Y.~Zhou,
\sl Stochastic Controls: Hamiltonian Systems and HJB Equations,
\rm Springer-Verlag, New York, 1999.

\bibitem{You 1983} Y.~You,
\it Optimal control for linear system with quadratic indefinite criterion on Hilbert spaces,
\sl Chin. Ann. Math. Ser. B, \rm 4 (1983), 21--32.

\bibitem{Zhou-Li 2000} X.~Y.~Zhou and D.~Li,
\it Continuous-time mean-variance portfolio selection: A stochastic LQ framework,
\sl Appl. Math. Optim., \rm 42 (2000), 19--33.


\end{thebibliography}
\end{document}